\documentclass[reqno,12pt,oneside]{amsart}

\usepackage{amsthm,amsmath,amssymb,amsfonts,accents}
\usepackage[colorlinks=true,hypertexnames=false,linkcolor=blue,citecolor=blue]{hyperref}
\usepackage[top=3cm,bottom=2.5cm,left=2.5cm, right=3cm,marginparwidth=0cm]{geometry}
\usepackage{epsfig,graphics}
\usepackage{amscd}
\usepackage{amsmath}
\usepackage{mathtools}
\usepackage{bm}
\usepackage{amssymb}
\usepackage{graphicx}
\usepackage{psfrag}
\usepackage{verbatim}
\usepackage{tikz}
\usetikzlibrary{quotes}
\usetikzlibrary{topaths}
\usetikzlibrary{shapes,arrows,cd}
\usepackage{adjustbox}
\usepackage{array}
\usepackage{textcomp}
\usepackage{gensymb}
\usepackage{enumitem}
\usepackage{csquotes}
\usepackage{color}
\usepackage{xcolor}
\usepackage{nameref}
\usepackage{chngcntr}
\usepackage{thmtools}
\usepackage{thm-restate}
\usepackage{import}
\usepackage{standalone}
\usepackage{tikz}
\usepackage{tikzscale}
\usepackage{tikz-cd}
\usepackage{pgfplots}
\usepackage{pgfmath}
\usepackage{float}

\usetikzlibrary{arrows, math}
 
\pgfplotsset{compat=1.16}

\newtheorem{theorem}{Theorem}[section]
\newtheorem{lemma}[theorem]{Lemma}
\newtheorem{prop}[theorem]{Proposition}
\newtheorem{coro}[theorem]{Corollary}

\theoremstyle{definition}
\newtheorem{definition}[theorem]{Definition}

\theoremstyle{remark}
\newtheorem{remark}[theorem]{Remark}

\theoremstyle{plain}

\numberwithin{equation}{section}

\theoremstyle{plain}

\theoremstyle{plain}

\newcommand{\ttt}[1]{\text{#1}}

\newcommand{\N}{\ensuremath{\mathbb{N}}}
\newcommand{\Z}{\ensuremath{\mathbb{Z}}}
\newcommand{\Q}{\ensuremath{\mathbb{Q}}}
\newcommand{\R}{\ensuremath{\mathbb{R}}}

\newcommand{\listset}[1]{
	\left\lbrace #1 \right\rbrace
}
\newcommand{\setbuilder}[2]{
    \left\lbrace \, #1 \; \middle| \; #2 \, \right\rbrace
}

\newcommand{\eqdef}{\coloneqq}

\DeclarePairedDelimiterX{\floor}[1]{\delimsize\lfloor}{\delimsize\rfloor}{#1}

\newcommand{\Id}{\operatorname{Id}}

\newcommand{\Leb}{\operatorname{Leb}}
\newcommand{\absol}[1]{\left\lvert #1 \right\rvert}
\newcommand{\normof}[1]{\left\| #1 \right\|}
\newcommand{\innerprod}[2]{\left\langle #1, #2 \right\rangle}

\newcommand{\varia}{\operatorname{var}}

\newcommand{\closure}[2]{\operatorname{cl}_{#2} \, #1}

\newcommand{\Distr}[2]{
	\mathcal{D}'_{#2} \left( #1 \right)
}
\newcommand{\cobound}{\operatorname{B}}

\newcommand{\crit}[1]{\operatorname{Crit}\left(#1\right)}

\newcommand{\orbcrit}[1]{\operatorname{Crit}^{\pm}\left(#1\right)}

\newcommand{\rat}{\omega}
\newcommand{\chum}[1]{\widetilde{#1}}
\newcommand{\chdois}[1]{\widehat{#1}}
\newcommand{\auto}[1]{\mathcal{A}_{#1}}

\newcommand{\ratcomp}{\sim}

\begin{document}

\title{Automorphic measures and invariant distributions for circle dynamics}

\author{Edson de Faria}
\address{Instituto de Matem\'atica e Estat\'istica, Universidade de S\~ao Paulo}
\curraddr{Rua do Mat\~ao 1010, 05508-090, S\~ao Paulo SP, Brasil}
\email{edson@ime.usp.br}

\author{Pablo Guarino}
\address{Instituto de Matem\'atica e Estat\'istica, Universidade Federal Fluminense}
\curraddr{Rua Prof. Marcos Waldemar de Freitas Reis, S/N, 24.210-201, Bloco H, Campus do Gragoat\'a, Niter\'oi, Rio de Janeiro RJ, Brasil}
\email{pablo\_\,guarino@id.uff.br}

\author{Bruno Nussenzveig}
\address{Instituto de Matem\'atica e Estat\'istica, Universidade de S\~ao Paulo}
\curraddr{Rua do Mat\~ao 1010, 05508-090, S\~ao Paulo SP, Brasil}
\email{bnussen@usp.br}

\thanks{The first author has been supported by ``Projeto Tem\'atico Din\^amica em Baixas Dimens\~oes'' FAPESP Grant  2016/25053-8, the second author has been partially supported by Conselho Nacional de Desenvolvimento Cient\'ifico e Tecnol\'ogico (CNPq), and the third author has been supported by ``Aplica\c{c}\~oes Multicr\'iticas do C\'irculo e Distribui\c{c}\~oes Invariantes'' FAPESP Grant 2021/04599-0.}

\subjclass[2010]{Primary 37E10; Secondary 37C40}

\keywords{Critical circle maps, automorphic measures, invariant distributions, Denjoy-Koksma inequality}

\begin{abstract} Let $f$ be a $C^{1+\ttt{bv}}$ circle diffeomorphism with irrational rotation number. As established by Douady and Yoccoz in the eighties, for any given $s>0$ there exists a unique \emph{automorphic measure} of exponent $s$ for $f$. In the present paper we prove that the same holds for \emph{multicritical circle maps}, and we provide two applications of this result. The first one, is to prove that the space of \emph{invariant distributions of order $1$} of any given multicritical circle map is one-dimensional, spanned by the unique invariant measure. The second one, is an improvement over the \emph{Denjoy-Koksma inequality} for multicritical circle maps and absolutely continuous observables.
\end{abstract}  

\maketitle

\vspace{-0.5cm}

\section{Introduction} \label{intro}

Smooth one-dimensional dynamical systems can be studied from various viewpoints, such as their topological classification, their smooth rigidity properties, the behaviour of their individual orbits or their measure-theoretic and ergodic properties. A specific class of such systems that has received a great deal of attention in recent years is the class of \emph{multicritical circle maps}.

A \emph{multicritical circle map} is a $C^3$ circle homeomorphism $f: S^1\to S^1$ having $N \geq 1$ critical points (all of which are \emph{non-flat}, see \S \ref{prelim}). We are only interested in maps of this type having no periodic points, in other words, only in those maps that have irrational rotation number. 
The classification os such maps up to topological conjugacy goes back to Yoccoz \cite{Yo3}, who proved that they are always minimal, hence topologically equivalent to a rotation of the circle (see Theorem \ref{yoccoztheorem} below). The smooth rigidity of such maps -- including the preliminary step known as \emph{quasi-symmetric rigidity} -- has been the object of intense research in recent decades; it is by now fairly well-understood, at least in the unicritical case, thanks to the combined efforts of several mathematicians, see \cite{EdF,edsonETDS,dFdM1,dFdM2,GY,GMdM,GdM,KT,KY,yampolsky1,yampolsky2,yampolsky3,yampolsky4}, or the book \cite{dFG:book} and references therein (we note \emph{en passant} that, quite recently, some rigidity results for maps with \emph{more than one} critical point have been established, see \cite{EG2023,ESY2022,Yam2019} and the recent preprint \cite{GorYam2021}). The geometric behaviour of individual orbits of such maps was examined in the recent paper \cite{dFG1}.

From the measure-theoretic viewpoint, multicritical circle maps have also been studied in detail. Having irrational rotation number, they are uniquely ergodic. Their unique invariant probability measure was shown to be purely singular with respect to Lebesgue (Haar) measure by Khanin \cite{Kh1} (see also \cite{GS93}), and later it was shown to have zero Lyapunov exponent in \cite{dFG2}. In \cite{dFG3}, the authors went a bit further and showed that such maps do not admit even $\sigma$-finite absolutely continuous invariant measures. 

In the present paper, we are interested in further ergodic-theoretic properties of multicritical circle maps. In particular, we are interested in the question: ``Does a (minimal) multicritical circle map admit other invariant \emph{distributions} besides its unique invariant probability measure?". The analogous question in a more general setting seems to have been first asked by Katok (for a general reference, see ~\cite{Ka}). However, our main source of inspiration is the remarkable paper by Avila and Kocsard \cite{AK} in which they give a fairly complete answer to the corresponding question for smooth circle \emph{diffeomorphisms}. 

Here, we give a (partial) answer to the above question by relating it (following the paper \cite{NT}, by Navas and Triestino) to the question of existence and uniqueness of so-called \emph{automorphic measures} for multicritical circle maps -- a question to which we give here a full answer.

Given $s\in\mathbb{R}$, an \emph{automorphic measure} of exponent $s$ for $f$ is a Borel probability measure $\nu$ on $S^1$ whose pullback $f^\ast\nu$ is equivalent to $\nu$, with Radon-Nikodym derivative given by $(Df)^s$. This concept is the analogue, for real one-dimensional maps, of the concept of \emph{conformal measure} introduced by Sullivan in \cite{Su1} in the context of rational maps  -- which in turn is inspired by a similar notion introduced by Patterson \cite{Pat} and Sullivan himself \cite{Su2} in the context of Fuchsian and Kleinian groups.

The precise definition of automorphic measure is given in \S \ref{auto} (def. \ref{def automorphic measures}). In the eighties, it was proved by Douady and Yoccoz (but only published some years later, in~\cite{DY} -- see also \cite{dMP}) that, for every minimal $C^{1+\ttt{bv}}$ circle diffeomorphism and every real number $s$, there exists a unique automorphic measure of exponent $s$. In the present paper we prove the following result.

\begin{restatable}[Existence and uniqueness of automorphic measures]{maintheorem}{existunique} \label{exist-unique} Let $f$ be a multicritical circle map. For any given $s \geq 0$ there exists a unique automorphic measure of exponent $s$ for $f$. This measure has no atoms, is supported on the whole circle and it is ergodic under~$f$.
\end{restatable}

\subsection{Applications} In Section \ref{distr} of the present paper we provide two applications of Theorem \ref{exist-unique}, that we now describe.

As usual, a \emph{$1$-distribution} is a continuous linear functional defined on the space of $C^1$ real-valued functions of the circle (see \S \ref{prelim inv distr} for precise definitions). As a consequence of Theorem \ref{exist-unique}, we have the following result.

\begin{restatable}[No invariant distributions]{maintheorem}{nodistr} \label{no invar distr} Let $f$ be a multicritical circle map with irrational rotation number and unique invariant measure $\mu$. Then the space $\Distr{f}{1}$ of $f$-invariant distributions of order at most $1$ is spanned by $\mu$, that is,
\[
\Distr{f}{1} = \R\mu\,.
\]
\end{restatable}

In other words, $f$ admits no invariant distributions of order at most $1$ different from (a scalar multiple of) its unique invariant measure. The proof of Theorem \ref{no invar distr} will be given in \S \ref{prelim inv distr}, and will follow the approach of Navas and Triestino developed in ~\cite{NT} for $C^{1+\ttt{bv}}$-diffeomorphisms. We would like to remark that this approach deals with distributions of order at most $1$. Since multicritical circle maps are assumed to be $C^3$ smooth, it would be desirable to also rule out invariant distributions of order up to $3$. Unfortunately, we do not know how to do this. Moreover, if we consider $C^{\infty}$, or $C^{\omega}$, multicritical circle maps, we do not know how to deal with higher order distributions. Let us be more precise: a $C^{\infty}$ dynamical system is \emph{distributionally uniquely ergodic} if it admits a single invariant distribution (up to multiplication by a constant). In \cite{AK}, Avila and Kocsard proved that every $C^{\infty}$ circle diffeomorphism with irrational rotation number is distributionally uniquely ergodic. We believe that $C^{\infty}$ multicritical circle maps are distributionally uniquely ergodic too but, as already mentioned, we do not know how to prove that. Nevertheless, to the best of our knowledge, Theorem \ref{no invar distr} provides the first examples of dynamics with no invariant distributions of order at most $1$ outside the realm of flows and diffeomorphisms. Finally, we would like to remark that the non-flatness condition on each critical point of $f$ (see \S \ref{prelim} below) is crucial in order to prove Theorem \ref{no invar distr}. Indeed, the following holds.

\begin{restatable}{maintheorem}{examplesHall} \label{TeoCexamples} For any given irrational number $\rho\in(0,1)$ there exists a $C^{\infty}$ homeomorphism $f \colon S^1 \to S^1$, with rotation number $\rho(f)=\rho$, having invariant distributions of order $1$ (different from a scalar multiple of its unique invariant measure).
\end{restatable}

The examples of Theorem \ref{TeoCexamples} are those constructed by Hall in \cite{Hall}, see \S \ref{ProofThmC} for the details.

\medskip

As it turns out, Theorem \ref{exist-unique} implies the following improvement over the Denjoy-Koksma inequality for absolutely continuous observables.

\begin{restatable}[Improved Denjoy-Koksma]{maintheorem}{denjoykoksma} \label{denjoy-koksma} Let $f$ be a multicritical circle map with irrational rotation number $\rho$ and unique invariant measure $\mu$, and let $\phi \colon S^1 \to \R$ be absolutely continuous. If $\{q_n\}$ is the sequence of denominators for the rational approximations of $\rho$, we have that
\begin{equation*}
q_n\left\|\frac{1}{q_n}\sum_{i = 0}^{q_n - 1} \, \phi \circ f^i-\int_{S^1} \phi \, d\mu\right\|_{C^0(S^1)} \longrightarrow 0\quad\mbox{as $n$ goes to $\infty$.}
\end{equation*}
\end{restatable}

The proof of Theorem \ref{denjoy-koksma} will be given in \S \ref{ProofThmD}, following the very same lines as that of Avila and Kocsard \cite[Section 3]{AK} and Navas \cite[Section 2]{NavasIMRN23} for circle diffeomorphisms.

\medskip

The paper is organized as follows: in \S \ref{prelim}, we present a recap of some topics from the theory of critical circle maps. In \S \ref{auto}, we define automorphic measures of exponent $s \geq 0$ for $f$, and explore some of their elementary properties; we remark that, at that point in the paper, it is still not clear whether $f$-automorphic measures actually exist. This is done in \S \ref{exist}, where we show that $f$-automorphic measures indeed exist for all $s > 0$. In \S \ref{bounds}, we obtain fundamental bounds on $f$-automorphic measures, which may themselves be of interest in future works. In \S \ref{ergodic}, we use said bounds to show that any $f$-automorphic measure is ergodic with respect to $f$, and as a consequence, easily derive the uniqueness part of Theorem \ref{exist-unique}. Finally, in \S \ref{distr}, we prove that Theorem \ref{exist-unique} implies both Theorem \ref{no invar distr} and Theorem \ref{denjoy-koksma}, and we briefly explain the proof of Theorem \ref{TeoCexamples}.

\section{Multicritical circle maps} \label{prelim}

Let us now define the maps which are the main object of study in the present paper. We start with the notion of {\it non-flat critical point\/}.

\begin{definition}\label{nonflat} We say that a critical point $c$ of a one-dimensional $C^3$ map $f$ is \emph{non-flat} of degree $d > 1$ if there exists a neighborhood $W_c$ of the critical point and a $C^3$ diffeomorphism
$\phi_c \colon W_c \to \phi_c(W_c) \subset \R$ such that $\phi_c(c) = 0$ and, for all $x \in W_c$,
\[
	f(x) = f(c) + \phi_c(x) \, \absol{\phi_c(x)}^{d-1}\,.
\] 
\end{definition}

This local form easily implies the following estimate (see ~\cite[ch. 5]{dFG:book}).

\begin{prop} \label{prop 5.4 dFG21} Let $c$ be a non-flat critical point of degree $d$ of a one-dimensional $C^3$ map $f$. There exists an interval $U = U_c \subset W_c$ that contains $c$ such that, for any non-empty interval $J \subset U$ and $x \in J$,
\begin{equation} \label{eq prop 5.4 dFG21}
D f(x) \leq 3d\,\frac{\absol{f(J)}}{\absol{J}}\,,
\end{equation}where $|J|$ denotes the Euclidean length of an interval $J$.
\end{prop}

\begin{definition} \label{def:multicritic} A \emph{multicritical circle map} is an orientation-preserving $C^3$ circle homeomorphism having finitely many critical points, all of which are non-flat.
\end{definition}

Being an orientation-preserving circle homeomorphism, a multicritical circle map $f$ has a well defined rotation number. We will focus on the case that $f$ has no periodic orbits (\emph{i.e.}, $\rho(f) \notin \Q$). As it turns out, these maps have no wandering intervals. More precisely, we have the following fundamental result.

\begin{theorem}\label{yoccoztheorem} Let $f$ be a multicritical circle map with irrational rotation number $\rho$. Then $f$ is topologically conjugate to the rigid rotation $R_{\rho}$, i.e., there exists a homeomorphism $h \colon S^{1} \to S^{1}$ such that $h \circ f = R_{\rho} \circ h.$
\end{theorem}

Theorem \ref{yoccoztheorem} was proved by Yoccoz in \cite{Yo3}, see also \cite[ch. 6]{dFG:book}.

\subsection{The Koebe distortion principle} \label{sec Koebe}

Given two circle intervals $M\subset T\subset S^{1}$ with $M$ compactly contained in the interior of $T$ (written $M \Subset T$), we denote by $L$ and $R$ the two connected components of $T\setminus M$. The {\it space of $M$ inside $T$} is defined to be the number
\begin{equation} \label{eq space}
	\tau = \min \listset{\frac{\absol{L}}{\absol{M}}, \frac{\absol{R}}{\absol{M}}} \, .
\end{equation}

Given circle intervals $M, T$ with $M \Subset T$ and $k \geq 1$ such that $f^k \colon T \to f^k(T)$ is a $C^1$ diffeomorphism onto its image, one can bound the distortion of $f^k$ inside $M$ independently of $k$ as long as the intermediate images $T, f(T), \cdots, f^{k-1}(T)$ satisfy a mild summability condition and the space of $f^k(M)$ inside $f^k(T)$ is bounded from below independently of $k$. This is the content of the Koebe distortion principle, and, as one can expect, it is of fundamental importance in controlling the geometric behavior of large iterates of the map $f$.

\begin{lemma}[Koebe distortion principle] \label{Koebe} For each $\ell, \tau > 0$ and each multicritical circle map $f \colon S^1 \to S^1$ there exists a constant $K = K(\ell, \tau, f) > 1$ with the following property. If $k \geq 1$, $M \subset T \subset S^1$ are intervals, with $M$ compactly contained in the interior of $T$, are such that the intervals $T, f(T), \dots, f^{k-1}(T)$ contain no critical point of $f$,
\begin{equation} \label{eq1 Koebe}
	\sum_{j = 0}^{k-1} \, \absol{f^j(T)} \leq \ell
\end{equation}
and the space of $f^k(M)$ inside $f^k(T)$ is at least $\tau$, then
\begin{equation} \label{eq2 Koebe}
	K^{-1} \leq \frac{D f^k(x)}{D f^k(y)} \leq K \quad \ttt{for all } x, y \in M \, .
\end{equation}
\end{lemma}

A proof of the Koebe distortion principle can be found in ~\cite[p.~295]{dMvS}. 

\begin{remark} \label{mult intersec and summability} Given a family of intervals $\mathcal{F}$ on $S^{1}$ and a positive integer $m$, we say that $\mathcal{F}$ has {\it multiplicity of intersection 
at most $m$\/} if each $x\in S^{1}$ belongs to at most $m$ elements of $\mathcal{F}$. For our purposes, the following (elementary) observation relating the hypotheses of the Koebe distortion principle to multiplicity of intersection shall be crucial: if the family $T, f(T), \dots, f^{k-1}(T)$ has multiplicity of intersection at most $m$, then \eqref{eq1 Koebe} holds with $\ell = m$. This observation also holds in the context of arbitrary finite measures on $S^1$: if $\nu$ is a finite measure on the circle, $m \geq 1$ and $\mathcal{F}$ is a family of circle intervals with intersection multiplicity at most $m$, then
\begin{equation*}
	\sum_{I \in \mathcal{F}} \, \nu(I) \leq m \, \nu(S^1)\,.
\end{equation*}
\end{remark}

\subsection{Combinatorics and real bounds} \label{prelim comb real bounds}

Throughout this paper, $f \colon S^1 \to S^1$ will be a $C^3$ multicritical circle map with irrational rotation number. Furthermore, $N \geq 1$ shall be the number of critical points of $f$, $\crit{f} = \listset{c_1, \dots, c_N}$ shall be the set of critical points of $f$, and $d_1, \dots, d_N$ their corresponding criticalities.

Let $\rho$ be the rotation number of $f$. As we know, it has a infinite continued fraction expansion, say
\begin{equation*}
      \rho(f)= [a_{0} , a_{1} , \cdots ]=   
      \cfrac{1}{a_{0}+\cfrac{1}{a_{1}+\cfrac{1}{ \ddots} }} \ .
    \end{equation*}
Truncating the expansion at level $n-1$, we obtain a sequence
of fractions $p_n/q_n$ which are called the \emph{convergents} of the irrational $\rho$.
$$
\frac{p_n}{q_n}\;=\;[a_0,a_1, \cdots ,a_{n-1}]\;=\;\dfrac{1}{a_0+\dfrac{1}{a_1+\dfrac{1}{\ddots\dfrac{1}{a_{n-1}}}}}\ .
$$
Since each $p_n/q_n$ is the best possible approximation to $\rho$ by fractions with denominator at most $q_n$,
we have
\begin{equation*}
\text{If} \hspace{0.4cm} 0<q<q_n \hspace{0.3cm} \text{then} \hspace{0.3cm} \left|\rho-\frac{p_n}{q_n}\right|<\left|\rho-\frac{p}{q}\right|, 
\hspace{0.3cm} \text{ for any $p\in\Z$.}
\end{equation*} 
The sequence of numerators satisfies
\begin{equation*}
 p_0=0,  \hspace{0.4cm} p_{1}=1, \hspace{0.4cm} p_{n+1}=a_{n}p_{n}+p_{n-1} \hspace{0.3cm} \text{for $n \geq 1$}.
\end{equation*}
Analogously, the sequence of the denominators, which we call the \emph{return times}, satisfies 
\begin{equation*}
 q_{0}=1, \hspace{0.4cm} q_{1}=a_{0}, \hspace{0.4cm} q_{n+1}=a_{n}q_{n}+q_{n-1} \hspace{0.3cm} \text{for $n \geq 1$} .
\end{equation*}

For each point $x \in S^1$ and each non-negative integer $n$, let $I_{n}(x)$ be the closed interval with endpoints 
$x$ and $f^{q_n}(x)$ containing  $f^{q_{n+2}}(x)$ (note that $I_{n}(x)$ contains no other iterate $f^{j}(x)$ with $1 \leq j \leq q_{n}-1$). We write $I_{n}^{j}(x)=f^{j}(I_{n}(x))$ for all $j$ and $n$.

\begin{lemma}\label{lemapartition} For each $n\geq 0$ and each $x\in S^1$, the collection of intervals
\[
 \mathcal{P}_n(x)\;=\; \left\{f^i(I_n(x)):\;0\leq i\leq q_{n+1}-1\right\} \;\bigcup\; \left\{f^j(I_{n+1}(x)):\;0\leq j\leq q_{n}-1\right\} 
\]
is a partition of the unit circle ({\it modulo} endpoints), called the {\it $n$-th dynamical partition\/} associated to the point $x$.
\end{lemma}

For a proof of this lemma, see ~\cite[ch. 6]{dFG:book}. Note that, for each $n$, the partition $\mathcal{P}_{n+1}(x)$ is a (non-strict) refinement of $\mathcal{P}_{n}(x)$, while the partition 
$\mathcal{P}_{n+2}(x)$ is a strict refinement of $\mathcal{P}_{n}(x)$.  

\begin{theorem}[Real Bounds] \label{real bounds} There exists a constant $C = C(f) > 1$, depending only on $f$, such that the following holds for every critical point $c$ of $f$. For all $n \geq 0$ and for each pair of adjacent atoms $I, J \in \mathcal{P}_n(c)$ we have
\begin{equation} \label{eq real bounds}
C^{-1} \absol{J} \leq \absol{I} \leq C \absol{J} \, .
\end{equation}
\end{theorem}

Note that for a rigid rotation we have $|I_n|=a_{n+1}|I_{n+1}|+|I_{n+2}|$. If $a_{n+1}$ is large, then $I_n$ is much larger than $I_{n+1}$. Thus, even for rigid rotations, real bounds do not hold in general.

Theorem \ref{real bounds} was obtained by Herman \cite{He2}, based on estimates by \'Swi\c{a}tek \cite{Sw1}. A detailed proof can be found in~\cite[ch. 6]{dFG:book}.

Theorem \ref{real bounds} has the following consequence (see ~\cite[ch. 8]{dFG:book}).

\begin{lemma} \label{lemma Dfqn} There exists $C_1 = C_1(f) > 0$ such that, for each $x \in S^1$ and all $n \geq 0$, we have $D f^{q_n}(x) \leq C_1$.
\end{lemma}

Yet another consequence of the real bounds that will be useful in the present paper (see \S \ref{exist} below) is the following.

\begin{lemma}[Zero Lyapunov Exponent]\label{ZeroLyapExp} Let $f$ be a multicritical circle map with irrational rotation number and unique invariant measure $\mu$. Then $\log D f \in L^1(\mu)$ and
\[
\int_{S^1}\log Df\,d\mu=0\,.
\]
\end{lemma}

A proof of Lemma \ref{ZeroLyapExp} can be found in~\cite{dFG2} (see also \cite[section 8.3]{dFG:book}).

\subsubsection{On the notions of domination and comparability} \label{dom comp 1}

To simplify both the understanding of and future calculations involving the real bounds, we introduce the notions of domination and comparability {\it modulo} $f$.

Given two circle intervals $I, J$, we shall say that {\it $I$ dominates $J$ modulo $f$}, and write $I \geq J$, if there exists a constant $K > 1$ depending only on $f$ such that $\absol{J} \leq K \absol{I}$. If both $I \geq J$ and $J \geq I$ (i.e. if there is $K = K(f) > 1$ such that $K^{-1} \absol{I} \leq \absol{J} \leq K \absol{I}$), we shall say that {\it $I$ and $J$ are comparable modulo $f$} (and write $I \asymp J$).

Thus, Theorem \ref{real bounds} states precisely that adjacent atoms of a dynamical partition are always comparable.

Observe that neither domination or comparability are transitive relations: if we are given a domination chain $I_1 \geq I_2 \geq \cdots \geq I_k$, we can only say that $I_1 \geq I_k$ if the length $k$ of the chain is bounded by a constant that depends only on $f$ (and similarly for comparability).

\section{Automorphic measures} \label{auto}

In this section we define automorphic measures of non-negative exponent for multicritical circle maps. We further prove that they have full support on the circle (Proposition \ref{automorphic measure full support}) and are non-atomic (Lemma \ref{no atoms}).

\begin{definition}[Automorphic measures] \label{def automorphic measures} Let $s \geq 0$.
An {\it automorphic measure of exponent $s$ for $f$} (or $f$-automorphic measure of exponent $s$) is a Radon probability measure $\nu$ on $S^1$ such that, for all continuous functions $\phi \in C^0(S^1)$,
\begin{equation} \label{eq def automorphic}
    \int_{S^1} \, \phi \, d\nu = \int_{S^1} \, (\phi \circ f) \, \left( D f \right)^s \, d\nu \, .
\end{equation}
We denote the set of $f$-automorphic measures of exponent $s$ by $\auto{s}$.
\end{definition}

Equivalently (see Proposition \ref{extension L1} below), a Radon probability measure $\nu$ on $S^1$ is $f$-automorphic of exponent $s$ if, and only if, the pullback measure $f^\ast \nu$ is equivalent to $\nu$, with Radon-Nikodym derivative
\[
	\frac{d f^\ast \nu}{d\nu} = (D f)^s \, .
\]

Observe that we leave out the possibility of negative exponents, \emph{i.e.}, $s < 0$. Though automorphic measures of negative exponent make perfect sense and indeed always exist in the case of diffeomorphisms\footnote{As it happens, the case $s=-1$ is suitable to understand both the variation of the rotation number along generic $1$-parameter families of circle diffeomorphisms \cite{dMP}, as well as to build the tangent space of the set of $C^2$ diffeomorphisms with a given irrational rotation number \cite[Th\'eor\`eme 2]{DY}.}, they are significantly more difficult to work with in the critical case. Indeed, if $s < 0$, then $(D f)^s$ blows up at the critical points of $f$, so we cannot take for granted that $(\phi \circ f) \, \left( D f \right)^s$ will be $\nu$-integrable for any $\phi \in C^0(S^1)$ and Radon probability measure $\nu$ on $S^1$.

We further remark that the notion of automorphic measures makes perfect sense on any dimension, provided the one-dimensional derivative $Df(x)$ is replaced by the Jacobian of $f$ at $x$, \emph{i.e.}, the absolute value of the determinant of the matrix $Df(x)$. As we mentioned in the introduction, for complex one-dimensional systems this is exactly the same as the notion of \emph{conformal measure} introduced by Sullivan in \cite{Su1}. In the present paper, however, we will of course only treat the real one-dimensional case.

Finally, observe that, in the case $s = 0$, an $f$-automorphic measure of exponent $0$ is simply an $f$-invariant probability measure. Therefore, the case $s = 0$ is well understood, and Theorem \ref{exist-unique} in this case is precisely the statement that $f$ (just like any circle homeomorphism with irrational rotation number) is uniquely ergodic. Therefore, for the rest of this paper, we will focus on the case $s > 0$. Thus, let $s > 0$ and $\nu \in \auto{s}$ be fixed.

\begin{prop} \label{extension L1} For all $\phi \in L^1(\nu)$ and $n \geq 1$, $(\phi \circ f^n) (D f^n)^s \in L^1(\nu)$ and
\begin{equation} \label{eq extension L1}
    \int_{S^1} \, \phi \, d\nu = \int_{S^1} \, (\phi \circ f^n) (D f^n)^s \, d\nu \, .
\end{equation}
\end{prop}

\begin{proof}
Observe that \eqref{eq extension L1} holds trivially if $\phi$ is continuous, by applying \eqref{eq def automorphic} inductively. The extension to $L^1$ functions $\phi$ now follows from a standard argument from the theory or Radon measures, with the main difficulty being to show that $(\phi \circ f^n) (D f^n)^s \in L^1(\mu)$ for all $\phi \in L^1(\mu)$.
\end{proof}

\medskip

Note in particular that, under forward iteration, the $\nu$-measure of a Borel set $A \subset S^1$ behaves according to the following rule:
\begin{equation}\label{AutIteration}
\nu\big(f^n(A)\big)=\int_{A}\,(D f^n)^s \, d\nu\,,\quad\mbox{for all $n\in\mathbb{N}$.}
\end{equation}

\begin{prop} \label{automorphic measure full support} The measure $\nu$ is supported on the entire circle.
\end{prop}

\begin{proof} Since the pullback measure $f^\ast \nu$ is equivalent to $\nu$, $\nu$-null sets are mapped under $f$ into $\nu$-null sets (see \eqref{AutIteration} above). But, since $f$ is topologically conjugate to an irrational rotation, the positive orbit of an open interval eventually covers the whole circle, and then this interval must have positive $\nu$-measure.
\end{proof}

\medskip

We denote by 
\[
	\orbcrit{f} = \bigcup_{j = 1}^N \, \mathcal{O}_f(c_j)
\]
the union of the critical orbits of $f$, and its complement $S^1 \setminus \orbcrit{f}$ by $\Lambda$. Observe that $\Lambda$ is $f$-invariant and its complement is countable (but dense!).

\begin{lemma} \label{no atoms} The measure $\nu$ has no atoms. In particular, $\Lambda$ has full $\nu$-measure on the circle.
\end{lemma}

\begin{proof} Arguing by contradiction, suppose there is some $x_0 \in S^1$ such that $\nu(\listset{x_0}) = \delta > 0$, and note that \eqref{AutIteration} implies
\[
\nu(\listset{x_0})=\nu\big(\listset{f^{-n}(x_0)}\big)\,\big(Df^n(f^{-n}(x_0))\big)^s\quad\mbox{for all $n\in\mathbb{N}$.}
\]
In particular, $x_0$ cannot be in the forward orbit of any critical point of $f$. Moreover, since $\nu$ is a probability measure and $f$ has no periodic orbits,
\begin{equation*}
1 \geq \sum_{n = 0}^{\infty} \, \nu(\listset{f^{-n}(x_0)}) = \delta \, \sum_{n = 0}^{\infty} \frac{1}{\left( D f^n(f^{-n}(x_0)) \right)^s} \geq \delta \, \sum_{n = 0}^{\infty} \frac{1}{\left( D f^{q_n}(f^{-q_n}(x_0)) \right)^s}\,.
\end{equation*}
However, by Lemma \ref{lemma Dfqn}, we have $D f^{q_n}(f^{-q_n}(x_0)) \leq C_1$ for all $n \geq 0$. Thus we obtain
\[
1 \geq \delta \, \sum_{n = 0}^{\infty} \, C_1^{-s} = \infty,
\]
which is the desired contradiction.
\end{proof}

\section{Existence} \label{exist}

In this section we show that, for all $s > 0$, $\auto{s}$ is non-empty (the existence part of Theorem \ref{exist-unique}). For the entire section, $s$ shall be a fixed positive number. 

Let $\Sigma \colon S^1 \to [0, \infty]$ be defined by
\begin{equation} \label{sigma s}
    \Sigma(x) = \sum_{n = 0}^{\infty} \, \left( Df^n(x) \right)^s.
\end{equation}
Observe that if $f^k(x) \in \crit{f}$ for some $k \geq 0$, then $\Sigma(x) < \infty$, since it is just a finite sum. Therefore, there is a dense subset of $S^1$ (the union of the backward orbits of the critical points) on which $\Sigma$ is finite. However, as the following lemma shows, there are points on the circle where $\Sigma$ diverges.

\begin{lemma} \label{divergence almost everywhere} 
$\Sigma = \infty$ $\mu$-almost everywhere.
\end{lemma}

\begin{proof} We show that the set
\begin{equation} \label{eq1 div ae}
	A \eqdef \setbuilder{x \in \Lambda}{\Sigma(x) = \infty}
\end{equation}
has full $\mu$-measure. To do this, first observe that, for all $x \in S^1$,
\begin{equation*}
	\Sigma(x) = 1 + (D f(x))^s \, \Sigma(f(x))\,.
\end{equation*}
It follows that $A$ is $f$-invariant, so, by the ergodicity of $\mu$, it suffices to show that $\mu(A) > 0$.

For each $n \geq 0$, $\left(D f^n\right)^s \in C^2(S^1)$, so we may apply Jensen's inequality to obtain
\begin{equation*}
\log \left( \int_{S^1} \, \left(D f^n\right)^s \, d\mu \right) \geq \int_{S^1} \, \log \left(D f^n\right)^s \, d\mu = s \sum_{i = 0}^{n - 1} \, \int_{S^1} \, \log D f \circ f^i \, d\mu\,.
\end{equation*}
Since $\mu$ is $f$-invariant,
\begin{equation} \label{eq2 div ae}
    \log  \left( \int_{S^1} \, \left(D f^n\right)^s \, d\mu \right) \geq s \sum_{i = 0}^{n - 1}
    \int_{S^1} \, \log D f \, d\mu = 0\,,
\end{equation}
where we have used Lemma \ref{ZeroLyapExp}. Thus,
\[
\int_{S^1} \, \left( D f^n \right)^s \, d\mu \geq 1
\]
for all $n \geq 0$, which implies in particular that
\begin{equation} \label{eq3 div ae}
	\int_{S^1} \, \sum_{k = 0}^{n - 1} \, (D f^{q_k})^s \, d\mu \geq n
\end{equation}
for all $n \geq 1$.

We argue by contradiction. Suppose $\mu(A) = 0$; then $\Sigma$ must be finite $\mu$-almost everywhere. Now, for each $m \geq 1$, let
\begin{equation} \label{eq4 div ae}
X_m \eqdef \setbuilder{x \in S^1}{\Sigma(x) \leq m} \, .
\end{equation}
Since we are assuming that $\Sigma$ is finite $\mu$-almost everywhere,
\begin{equation} \label{eq5 div ae}
\lim_{m \to \infty} \, \mu(X_m) = 1.
\end{equation}

Let $0 < \epsilon < C_1^{-s}$, where $C_1 = C_1(f)$ is the constant of Lemma \ref{lemma Dfqn}. From \eqref{eq5 div ae}, there exists $m_0 \in \N$ such that, for all $m \geq m_0$, $\mu(S^1 \setminus X_m) \leq \epsilon$.

But then, from \eqref{eq3 div ae}, \eqref{eq4 div ae} and Lemma \ref{lemma Dfqn}, we have that, for all $n \geq 1$ and $m \geq m_0$,
\begin{equation} \label{eq6 div ae}
\begin{split}
n &\leq \int_{S^1} \, \sum_{k = 0}^{n - 1} \, (D f^{q_k})^s \, d\mu = \int_{X_m} \, \sum_{k = 0}^{n - 1} \, (D f^{q_k})^s \, d\mu + \int_{S^1 \setminus X_m} \, \sum_{k = 0}^{n - 1} \, (D f^{q_k})^s \, d\mu \\
&\leq \int_{X_m} \, \Sigma \, d\mu + \sum_{k = 0}^{n-1} \int_{S^1 \setminus X_m} \, (D f^{q_k})^s \, d\mu \leq m \, \mu(X_m) + n \, C_1^s \, \mu(S^1 \setminus X_m) \\
&\leq m + n \, C_1^s \, \epsilon\,,
\end{split}
\end{equation}
which implies that $m \geq n \, (1 - C_1^s \, \epsilon)$. Since $1 - C_1^s \, \epsilon > 0$, the contradiction arises by letting $n \to \infty$ while keeping $m$ fixed.
\end{proof}

\vspace{0.2cm}

As it turns out, the only thing we will need in the proof below from the set $A$ is the fact that it is non-empty, which certainly follows from Lemma~\ref{divergence almost everywhere}.

\vspace{0.2cm}

\begin{proof}[Proof of Theorem \ref{exist-unique}, existence part] Consider the function $\Sigma$ and the set $A$ from Lemma \ref{divergence almost everywhere}. Fix
$x \in A$ and, for each $n \geq 1$, let
\[
	S_n(x) = \sum_{i = 0}^{q_n - 1} \, (D f^i(x))^s\,.
\]
Consider
\[
	\mu_{s, x, n} \eqdef \frac{1}{S_n(x)} \, \sum_{i = 0}^{q_n - 1} \, (D f^i(x))^s \, \delta_{f^i(x)}\,,
\] 
which is an atomic probability measure. By compactness, there is a monotone sequence $(n_k) \subset \N$ and $\mu_{s,x} \in \mathcal{P}(S^1)$ such that, for all $\phi \in C^0(S^1)$,
\[
    \int_{S^1} \, \phi \, d\mu_{s,x,n_k} \longrightarrow \int_{S^1} \, \phi \, d\mu_{s,x}\,.
\]
In particular, since $(\phi \circ f) (D f)^s \in C^0(S^1)$,
\[
    \int_{S^1} \, (\phi \circ f) (D f)^s \, d\mu_{s,x,n_k} \longrightarrow \int_{S^1} \, (\phi \circ f) (D f)^s \, d\mu_{s,x}
\]for all $\phi \in C^0(S^1)$. We claim that $\mu_{s,x}$ is $s$-automorphic under $f$. Indeed, for all $k \geq 1$ and $\phi \in C^0(S^1)$, we have
\begin{equation*}
\begin{split}
    & \quad \absol{\int_{S^1} \, [\phi - (\phi \circ f) (D f)^s] \, d\mu_{s,x,n_k}} \\
    &= \frac{1}{S_{n_k}(x)} \, \absol{\sum_{i = 0}^{q_{n_k} - 1} \, (D f^i(x))^s \, [\phi(f^i(x)) - \phi(f^{i+1}(x)) (D f(f^i(x)))^s]} \\
    &= \frac{1}{S_{n_k}(x)} \, \absol{\phi(x) - \phi(f^{q_{n_k}}(x)) \, (D f^{q_{n_k}}(x))^s} \\
    &\leq \normof{\phi}_{C^0} \, [1 + (D f^{q_{n_k}}(x))^s] \, \frac{1}{S_{n_k}(x)} \leq \normof{\phi}_{C^0} (1 + C_1^s) \, \frac{1}{S_{n_k}(x)}\,,
\end{split}
\end{equation*}
where we have used Lemma \ref{lemma Dfqn}. Consequently,
\begin{equation*}
\begin{split}
    &\quad \absol{\int_{S^1} \, \phi \, d\mu_{s,x} - \int_{S^1} \, (\phi \circ f) (D f)^s \, d\mu_{s, x}} = \lim_{k \to \infty} \, \absol{\int_{S^1} \, [\phi - (\phi \circ f) (D f)^s] \, d\mu_{s,x,n_k}} \\
    &\leq \normof{\phi}_{C^0} (1 + C_1^s) \, \lim_{k \to \infty} \frac{1}{S_{n_k}(x)} = 0\,,
\end{split}
\end{equation*}
since $x \in A$. Thus, $\mu_{s,x}$ is $f$-automorphic of exponent $s$, which concludes the proof. \end{proof}

\begin{remark} \label{Yoccoz operator} In ~\cite{DY}, Douady and Yoccoz prove the existence part of Theorem \ref{exist-unique} in the context of diffeomorphisms through a different approach. First, they define a continuous operator $U_{s,f} \colon \mathcal{M}(S^1) \to \mathcal{M}(S^1)$ on the space $\mathcal{M}(S^1)$ of signed finite Radon measures on the circle (equipped with the weak-$\ast$ topology) by
\[
\int_{S^1} \, \phi \, d(U_{s, f} \nu) \eqdef \frac{1}{\int_{S^1} \, (D f)^s \, d\nu} \, \int_{S^1} \, (\phi \circ f)(D f)^s \, d\nu\,.
\]
Clearly, the operator $U_{s,f}$ leaves invariant the convex compact set $\mathcal{P}(S^1)$ of Radon probability measures on the circle. The authors then use the Schauder-Tychonoff fixed point theorem to conclude that $U_{s, f}$ has a fixed point $\mu_s \in \mathcal{P}(S^1)$, and through some estimates, they conclude that this fixed point $\mu_s$ must be $f$-automorphic of exponent $s$.

In the critical case, this approach fails. Indeed, if $\nu = \delta_c$ is a point mass on a critical point $c$ of $f$, then $U_{s, f} \nu$ is ill-defined, since 
\[
\int_{S^1} \, (D f)^s \, d\nu = (D f(c))^s = 0\,.
\]
Furthermore, if we remove from $\mathcal{P}(S^1)$ the point masses at the critical points of $f$, then we lose compactness, which is essential to apply the Schauder-Tychonoff fixed point theorem.
\end{remark}

\section{Bounds for automorphic measures} \label{bounds}

In this section, we dive further into the fine-scale structure of $f$-automorphic measures of exponent $s > 0$. For this entire section, fix some $s > 0$ and $\nu \in \auto{s}$, and let $B \subset S^1$ be an arbitrary Borel $f$-invariant set.

In what follows, the ratio
\begin{equation} \label{eq ratio 1}
	\frac{\nu(I \cap B)}{\absol{I}^s}
\end{equation}
where $I$ is an interval, will play a fundamental role. Hence we introduce the special notation:
\begin{equation} \label{eq ratio 2}
	\rat(I) \eqdef \frac{\nu(I \cap B)}{\absol{I}^s} \, .
\end{equation}

The following theorem is the main result of this section.

\begin{restatable}{theorem}{gencomp} \label{general comparability of ratios} There exists a constant $B = B(f, s) > 1$ with the following property. For any critical point $c$ of $f$, sufficiently large $n$ and $\Delta_1, \Delta_2 \in \mathcal{P}_n(c)$, we have:
\begin{enumerate}
    \item[(a)] If $\Delta_1, \Delta_2$ are both long atoms or both short atoms of $\mathcal{P}_n(c)$, then
    \begin{equation} \label{eq comp ratios same type}
    B^{-1} \, \rat(\Delta_2) \leq \rat(\Delta_1) \leq B \, \rat(\Delta_2) \, .
    \end{equation}
    \item[(b)] If $\Delta_1$ is a short atom and $\Delta_2$ is a long atom of $\mathcal{P}_n(c)$, then
    \begin{equation} \label{eq comp ratios diff type}
        \rat(\Delta_1) \leq B \, \rat(\Delta_2) \, .
    \end{equation}
\end{enumerate}
\end{restatable}

\subsection{Fundamental estimates on distortion} \label{distortion-estimates}

We must now introduce a bit of notation. For the rest of this paper, we fix a critical point $c$ of $f$, and we write simply $\mathcal{P}_n$ in place of $\mathcal{P}_n(c)$. Furthermore, if $I \subset S^1$ is an interval, we write $I^k$ for $f^k(I)$. For any $n \geq 0$ and any atom $\Delta \in \mathcal{P}_n$, we write $\Delta^\ast$ for the reunion of $\Delta$ with its two adjacent atoms, $L$ and $R$, in $\mathcal{P}_n$. For example, if $\Delta = I_n$, then
\[
	\Delta^\ast = I_{n+1} \cup I_n \cup I_n^{q_n} \, .
\]
We also write $\chum{\Delta}$ for the following interval. First write $\Delta$ as a reunion of atoms of $\mathcal{P}_{n+2}$ and let $L_1, R_1$ be the leftmost and rightmost atoms of $\mathcal{P}_{n+2}$ in this reunion, respectively; we then take $\chum{\Delta} = L_2 \cup \Delta \cup R_2$, where $L_2, R_2$ are the atoms of $\mathcal{P}_{n+2}$ left-adjacent to $L_1$ and right-adjacent to $R_1$, respectively. For example, if $\Delta = I_n$, then
\begin{equation} \label{delta star}
\chum{\Delta} = I_{n+3} \cup I_n \cup I_{n+2}^{q_n} \, .
\end{equation}

Lastly, we write $\chdois{\Delta}$ for the following interval. If $\Delta^\ast = L \cup \Delta \cup R$, $L^\ast = (L)_2 \cup L \cup \Delta$ and $R^\ast = \Delta \cup R \cup (R)_2$, we shall write
\begin{equation} \label{delta star star}
    \chdois{\Delta} = (L)_2^\ast \cup \Delta \cup (R)_2^\ast = (L)_3 \cup (L)_2 \cup L \cup \Delta \cup R \cup (R)_2 \cup (R)_3 \supset \Delta^\ast \, .
\end{equation}
For example, if $\Delta = I_n$ and $a_n \geq 5$, then
\[
	\chdois{\Delta} = I_n^{q_{n+1} - 2 q_n} \cup I_n^{q_{n+1} - q_n} \cup I_{n+1} \cup I_n \cup I_n^{q_n} \cup I_n^{2 q_n} \cup I_n^{3 q_n} \, .
\]

\begin{figure}[!ht]
    \centering
    \begin{tikzpicture}
    
    \begin{scope}[thick]
    
    \draw (-7, 2) -- (7, 2);
    \draw (-3, 1) -- (3, 1);
    \draw (-1.4, 0) -- (1.4, 0);
    
    \end{scope}
    
    \draw (0, 2.1) node[anchor=south] {$\Delta$};
    \draw (0, 1.1) node[anchor=south] {$\Delta$};
    \draw (0, 0.1) node[anchor=south] {$\Delta$};
    \draw (-2, 2.1) node[anchor=south] {$L$};
    \draw (-2, 1.1) node[anchor=south] {$L$};
    \draw (2, 2.1) node[anchor=south] {$R$};
    \draw (2, 1.1) node[anchor=south] {$R$};
    \draw (-4, 2.1) node[anchor=south] {$(L)_2$};
    \draw (4, 2.1) node[anchor=south] {$(R)_2$};
    \draw (-6, 2.1) node[anchor=south] {$(L)_3$};
    \draw (6, 2.1) node[anchor=south] {$(R)_3$};
    \draw (-1.2, 0.1) node[anchor=south] {$L_2$};
    \draw (1.2, 0.1) node[anchor=south] {$R_2$};
    \draw (8, 2) node[anchor=west] {$\chdois{\Delta}$};
    \draw (4, 1) node[anchor=west] {$\Delta^\ast$};
    \draw (2, 0) node[anchor=west] {$\chum{\Delta}$};
    
    \filldraw[color=black] (-7, 2) circle (0.05cm);
    \filldraw[color=black] (-5, 2) circle (0.05cm);
    \filldraw[color=black] (-3, 2) circle (0.05cm);
    \filldraw[color=black] (-1, 2) circle (0.05cm);
    \filldraw[color=black] (1, 2) circle (0.05cm);
    \filldraw[color=black] (3, 2) circle (0.05cm);
    \filldraw[color=black] (5, 2) circle (0.05cm);
    \filldraw[color=black] (7, 2) circle (0.05cm);
    \filldraw[color=black] (-3, 1) circle (0.05cm);
    \filldraw[color=black] (-1, 1) circle (0.05cm);
    \filldraw[color=black] (1, 1) circle (0.05cm);
    \filldraw[color=black] (3, 1) circle (0.05cm);
    \filldraw[color=black] (-1.4, 0) circle (0.05cm);
    \filldraw[color=black] (-1, 0) circle (0.05cm);
    \filldraw[color=black] (1, 0) circle (0.05cm);
    \filldraw[color=black] (1.4, 0) circle (0.05cm);
    
	\end{tikzpicture}
	\caption{The intervals $\Delta^\ast$, $\chum{\Delta}$ and $\chdois{\Delta}$.}
    \label{fig: delta ast chum and chdois}
\end{figure}
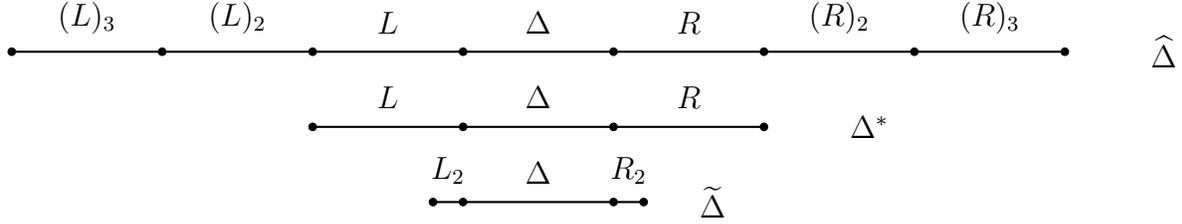

Of course, if $n$ is small, it may be that $\mathcal{P}_n$ has at most $7$ atoms, so in this case we would have $\chdois{\Delta} = S^1$. Thus, when dealing with $\chdois{\Delta}$, we alwaus assume that $n$ is sufficiently large for $\chdois{\Delta}$ to be a proper interval.

To provide the bounds on distortion needed for the rest of this paper, we shall need the following combinatorial facts. Although their proofs are somewhat involved, the techniques used are standard. Accordingly, we have decided to omit these proofs.

Recall from \S \ref{prelim} that, given a family of intervals $\mathcal{F}$ on $S^{1}$ and a positive integer $m$, we say that $\mathcal{F}$ has multiplicity of intersection at most $m$ if each $x\in S^{1}$ belongs to at most $m$ elements of $\mathcal{F}$. 

\begin{lemma} \label{bounding mult intersec} Let $n \geq 0$, $\Delta \in \mathcal{P}_n$. Then:
\begin{enumerate}
\item[(a)] the collection $\listset{f^k(\Delta^\ast)}_{k=0}^{q_{n+1}-1}$ has intersection multiplicity at most $3$;
\item[(b)] the collection $\listset{f^k(\chum{\Delta})}_{k=0}^{q_{n+1}-1}$ has intersection multiplicity at most $3$;
\item[(c)] the collection $\listset{f^k(\chdois{\Delta})}_{k=0}^{q_{n+1}-1}$ has intersection multiplicity at most $8$.
\end{enumerate}
\end{lemma}

We shall need the following consequence of the Real Bounds.

\begin{lemma} \label{comparability of a bunch of images of atoms} There exists a constant $C_2 = C_2(f) \geq C > 1$ with the following property. For $n \geq 0$, let
\begin{equation*} 
\mathcal{C}_n \eqdef \listset{I_n^j}_{j = 0}^{2 q_{n+1}} \cup \listset{I_{n+1}^k}_{k = 0}^{q_n + q_{n+1}}
\end{equation*}
be the set of all atoms of $\mathcal{P}_n$, together with their forward images under $f$ up to iterate $q_{n+1} + 1$. Then, for any $J_1, J_2 \in \mathcal{C}_n$ that share a common endpoint,
\begin{equation} \label{comp bunch}
C_2^{-1} \absol{J_1} \leq \absol{J_2} \leq C_2 \absol{J_1} \, .
\end{equation}
\end{lemma}

The following two lemmas contain the bounds on distortion needed for the rest of this paper. Their proof is a standard application of Koebe's distortion principle, with Lemmas \ref{bounding mult intersec} and \ref{comparability of a bunch of images of atoms} guaranteeing that the corresponding hypotheses on summability and space are satisfied (recall Section \ref{sec Koebe}).

\begin{lemma} \label{Koebe for measure} There exists $B_0 = B_0(f) > 1$ with the following property.
If $\Delta \in \mathcal{P}_n$ and $0 \leq j < k \leq q_{n+1} + 1$ are such that the intervals $f^j(\chum{\Delta}), f^{j+1}(\chum{\Delta}), \dots, f^{k-1}(\chum{\Delta})$ do not contain any critical point of $f$, then the map $f^{k-j} \colon f^j(\Delta) \to f^k(\Delta)$ has distortion bounded by $B_0$, that is
\[
    B_0^{-1} \leq \frac{Df^{k-j}(x)}{Df^{k-j}(y)} \leq B_0 \quad\quad \ttt{for all } x, y \in f^j(\Delta) \, .
\]
\end{lemma}

\begin{lemma} \label{Koebe for measure 2} There exists $B_1 = B_1(f) > 1$ with the following property.
If $\Delta \in \mathcal{P}_n$ and $0 \leq j < k \leq q_{n+1}$ are such that the intervals $f^j(\chdois{\Delta}), f^{j+1}(\chdois{\Delta}), \dots, f^{k-1}(\chdois{\Delta})$ do not contain the critical point of $f$, then the map $f^{k-j} \colon f^j(\Delta^\ast) \to f^k(\Delta^\ast)$ has distortion bounded by $B_1$, that is
\[
    B_1^{-1} \leq \frac{Df^{k-j}(x)}{Df^{k-j}(y)} \leq B_1 \quad\quad \ttt{for all } x, y \in f^j(\Delta^\ast)  \, .
\]
\end{lemma}

We remark that $\chdois{\Delta}$ and Lemma \ref{Koebe for measure 2} will not be mentioned further in this section, but will play a fundamental role in \S \ref{ergodic}.

\subsection{$\rat$-domination and comparability} \label{rat dom comp}

To simplify both the statement and the proof of the remaining results in this section, we introduce the notions of {\it $\rat$-domination} and {\it $\rat$-com\-pa\-ra\-bi\-li\-ty} between intervals. If $I, J \subset S^1$ are intervals, we shall say that {\it $I$ $\rat$-dominates $J$} (and write $I \succeq J$) if there is some constant $K = K(f, s) > 1$ (depending {\it only} on $f$ and $s$, but not on $\nu$ or $B$) such that
\begin{equation} \label{eq ratio domination}
\rat(J) \leq K \rat(I).
\end{equation}
Similarly, we say that $I, J$ are {\it $\rat$-comparable} (and write $I \ratcomp J$) if $I \succeq J$ and $J \succeq I$; that is, if there is some constant $K = K(f, s) > 1$ such that
\begin{equation} \label{eq ratio comparability}
K^{-1} \rat(I) \leq \rat(J) \leq K \rat(I).
\end{equation}

\begin{definition} \label{critical time of type 1} Let $\Delta \in \mathcal{P}_n$, $0 \leq k < q_{n+1}$. We shall say that $k$ is a {\it critical time of type 1} for $\Delta$ if $f^k(\chum{\Delta}) \cap \crit{f} \neq \varnothing$.
\end{definition}

Since $f$ has $N$ critical points $c_1, \dots, c_N$ and the collection $\listset{f^k(\chum{\Delta})}_{k=0}^{q_{n+1}-1}$ has intersection multiplicity at most $3$ (see Lemma \ref{bounding mult intersec}), it follows that, for any $n \geq 0$ and $\Delta \in \mathcal{P}_n$, there are at most $3N$ critical times of type 1 for $\Delta$. 

\begin{remark} \label{one critical point from critical time} It follows easily from the minimality of $f$ that there exists some level $n_0 = n_0(f) \in \N$, depending only on $f$, such that, for all $n \geq n_0$, $\Delta \in \mathcal{P}_n$ and $0 \leq k < q_{n+1}$ a critical time of type 1 for $\Delta$, we have that: (i) $f^k(\chum{\Delta})$ contains a {\it single} critical point of $f$; and (ii) $f^k(\chum{\Delta}) \subset U$, where $U$ is the interval about the critical point of $f$ in $f^k(\chum{\Delta})$ from Proposition \ref{prop 5.4 dFG21}.
\end{remark}

The following lemma tells us what happens to the ratios $\rat(I)$ as we iterate $f$ while staying (combinatorially) far away from the critical points of $f$.

\begin{lemma} \label{comparability of ratios 1} Let $n \geq 0$ and $\Delta \in \mathcal{P}_n$. Then, for any interval $I \subset \Delta$ and $0 \leq j < k \leq q_{n+1} + 1$ such that the intervals $f^j(\chum{\Delta}), \dots, f^{k-1}(\chum{\Delta})$ do not contain any critical point of $f$,
\begin{equation} \label{eq comp ratios 1}
B_0^{-s} \, \rat(I^j) \leq \rat(I^k) \leq B_0^s \, \rat(I^j) \, .
\end{equation}
\end{lemma}

\begin{proof} Indeed, by Lemma \ref{Koebe for measure}, the distortion of $f^{k-j}$ in $\Delta^j$ is bounded by $B_0$. By the Mean Value Theorem, there exists $z \in I^j$ such that
\[
    D f^{k-j}(z) = \frac{\absol{I^k}}{\absol{I^j}} \, .
\]
Thus, for any $x \in \Delta^j \supset I^j$,
\begin{equation} \label{eq1 comp ratios 1}
    B_0^{-1} \frac{\absol{I^k}}{\absol{I^j}} \leq D f^{k-j}(x) \leq B_0 \frac{\absol{I^k}}{\absol{I^j}} \, .
\end{equation}

From equation \eqref{AutIteration}, we have
\begin{equation*}
    \rat(I^k) = \absol{I^k}^{-s} \, \int_{I^j \cap B} \, (D f^{k-j}(x))^s \, d\nu(x)
\end{equation*}
so, from \eqref{eq1 comp ratios 1}, we get
\[
    B_0^{-s} \frac{\absol{I^k}^s}{\absol{I^j}^s} \absol{I^k}^{-s} \nu(I^j \cap B) \leq \rat(I^k) \leq B_0^s \frac{\absol{I^k}^s}{\absol{I^j}^s} \absol{I^k}^{-s} \nu(I^j \cap B)
\]
which is \eqref{eq comp ratios 1}.
\end{proof}

The next result is now an easy corollary of Lemma \ref{comparability of ratios 1}:

\begin{coro} \label{comparability of ratios between critical times} Let $n \geq 0$, $\Delta \in \mathcal{P}_n$, and let $0 \leq k_1 < k_2 < \cdots < k_r < q_{n+1}$  be the critical times of type 1 for $\Delta$. Then, for any interval $I \subset \Delta$,
\begin{enumerate}
\item[(a)] $I, I^1, \dots, I^{k_1}$ are pairwise $\rat$-comparable;
\item[(b)] For $1 \leq j < r$, $I^{k_j+1}, I^{k_j+2}, \dots, I^{k_{j+1}}$ are pairwise $\rat$-comparable;
\item[(c)] $I^{k_r+1}, \dots, I^{q_{n+1}}$ are pairwise $\rat$-comparable,
\end{enumerate}
all with constant $K(f, s) = B_0^s$.
\end{coro}

To proceed, we need to study the behaviour of $\rat$ close to the critical set of $f$. For this purpose, recall that $d > 1$ denotes the maximum of the criticalities of the critical points of $f$.

\begin{lemma} \label{critical times dominate} Let $n \geq n_0$ (from Remark \ref{one critical point from critical time}), $\Delta \in \mathcal{P}_n$, and let $I \subset \Delta$ be an interval.
\begin{enumerate}
\item[(a)] If $0 \leq j < q_{n+1}$ is a critical time of type 1 for $\Delta$, $I^j \succeq I^{j+1}$ with constant $K(f, s) = (3 d)^s$;
\item[(b)] if $0 \leq \ell_1 < \ell_2 \leq q_{n+1}$, $I^{\ell_1} \succeq I^{\ell_2}$, with constant $K(f, s) = (3 d B_0)^{4 N s}$.
\end{enumerate}
\end{lemma}

\begin{proof} Observe that part (a) and Corollary \ref{comparability of ratios between critical times} together imply (b), since we can join $I^{\ell_1}$ and $I^{\ell_2}$ by a $\rat$-domination chain as follows: let $\ell_1 \leq k_i < \cdots < k_m < \ell_2$ be the critical times of type $1$ for $\Delta$ between $\ell_1$ and $\ell_2$. Then
\[
	I^{\ell_1} \ratcomp I^{k_i} \succeq I^{k_i + 1} \ratcomp I^{k_{i+1}} \succeq \cdots \ratcomp I^{k_\ell} \succeq I^{k_m + 1} \ratcomp I^{\ell_2} \, .
\]
Since there are at most $6 N + 2$ atoms in this chain, it follows that $I^{\ell_1} \succeq I^{\ell_2}$. To determine the constant of $\rat$-domination, we start with $K(f, s) = 1$ and move along this chain, multiplying by $B_0^s$ for every $\ratcomp$ and by $(3 d)^s$ for every $\succeq$. There are $m - i + 2$ $\ratcomp$'s and $m - i + 1$ $\succeq$'s, so (since $m - i + 1\leq 3N$) we can take
\[
	K(f, s) = (3 d B_0)^{4 N s} \geq (3 d)^{(m - i + 1) s} B_0^{(m - i + 2) s} \, .
\]
Thus, we need only prove (a).

Observe that, from equation \eqref{AutIteration},
\begin{equation} \label{eq1 dom crit}
\rat(I^{j+1}) = \frac{\int_{I^j \cap B} \, (D f)^s \, d\nu}{\absol{I^{j+1}}^s} = \frac{\absol{I^j}^s}{\absol{I^{j+1}}^s} \frac{\int_{I^j \cap B} \, (D f)^s \, d\nu}{\absol{I^j}^s} \, .
\end{equation}

Now, Proposition \ref{prop 5.4 dFG21} implies that
\begin{equation} \label{eq2 dom crit}
(D f(x))^s \leq (3 d)^s \frac{\absol{I^{j+1}}^s}{\absol{I^j}^s} \quad \ttt{for all } x \in I^j \, .
\end{equation}
Combining \eqref{eq1 dom crit} and \eqref{eq2 dom crit}, we get
\begin{equation} \label{eq3 dom crit}
\rat(I^{j+1}) \leq (3 d)^s \rat(I^j)
\end{equation}
which proves (a).
\end{proof}

Before moving forward to the proof of Theorem \ref{general comparability of ratios}, we shall first need a lemma which states, essentially, that {\it long atoms $\rat$-dominate short atoms}. To simplify the proofs of this lemma and of Theorem \ref{general comparability of ratios} below, we shall denote all constants of $\omega$-domination by $K = K(f, s)$.

\begin{lemma} \label{comparability of ratios short-long} Let $n \geq n_0$, $\Delta_1, \Delta_2 \in \mathcal{P}_n \cup \listset{I_n^{q_{n+1}}, I_{n+1}^{q_n}}$. If $\Delta_1$ is a long atom  (or $I_n^{q_{n+1}}$) and $\Delta_2$, a short atom (or $I_{n+1}^{q_n}$), of $\mathcal{P}_n$, then $\Delta_1 \succeq \Delta_2$. Furthermore, if $a_{n+1} \geq 2$ or $a_{n+1} = a_{n+2} = 1$, then 
$\Delta_1^{k_1} \succeq \Delta_2^{k_2}$ for any $0 \leq k_1, k_2 \leq q_{n+1}$. 
\end{lemma}

\begin{proof} We split the proof in two parts: (i) that $\Delta_1 \succeq \Delta_2$; and (ii) that $\Delta_1^{k_1} \succeq \Delta_2^{k_2}$ if $a_{n+1} \geq 2$ or $a_{n+1} = a_{n+2} = 1$.

To prove (i), observe that, from Lemma \ref{critical times dominate}, we have $\Delta_1 \succeq I_n^{q_{n+1}}$ and $I_{n+1} \succeq \Delta_2$. Therefore, (i) will follow if we prove that $I_n^{q_{n+1}} \succeq I_{n+1}$. Since $I_n^{q_{n+1}} \supset I_{n+1}$ (so $\nu(I_{n+1}) \leq \nu(I_n^{q_{n+1}})$), this is a consequence of the fact that these two intervals have comparable lengths (see ~\cite[Prop. 6.1]{dFG:book}).

We now turn to the proof of (ii). From Lemma \ref{critical times dominate}, we get $\Delta_1^{k_1} \succeq I_n^{2 q_{n+1} - 1}$ and $I_{n+1} \succeq \Delta_2^{k_2}$. By applying either Lemma \ref{comparability of ratios 1} or Lemma \ref{critical times dominate}, depending on whether $I_n^{2 q_{n+1} - 1} \cap \crit{f} = \varnothing$ or not, we get $I_n^{2 q_{n+1} - 1} \succeq I_n^{2 q_{n+1}}$, from which it follows that $\Delta_1^{k_1} \succeq I_n^{2 q_{n+1}}$. Thus, it suffices to show that $I_n^{2 q_{n+1}} \succeq I_{n+1}$.

We first consider the case $a_{n+1} \geq 2$. For this end, observe from figure \ref{fig: Last elements of Cn} that
\begin{equation} \label{eq1 comp short long}
I_n^{2 q_{n+1}} = I_{n+1}^{q_{n+1}} \cup I_{n+1} \cup [I_n \setminus (I_{n+1}^{q_n} \cup I_{n+1}^{q_n+q_{n+1}})]
\end{equation}
with the unions disjoint {\it modulo} endpoints. Thus, $I_n^{2 q_{n+1}} \supset I_{n+1}$, which implies that
\begin{equation} \label{eq2 comp short long}
\rat(I_{n+1}) \leq \left(\frac{\absol{I_n^{2 q_{n+1}}}}{\absol{I_{n+1}}}\right)^s \rat(I_n^{2 q_{n+1}}) \, .
\end{equation}

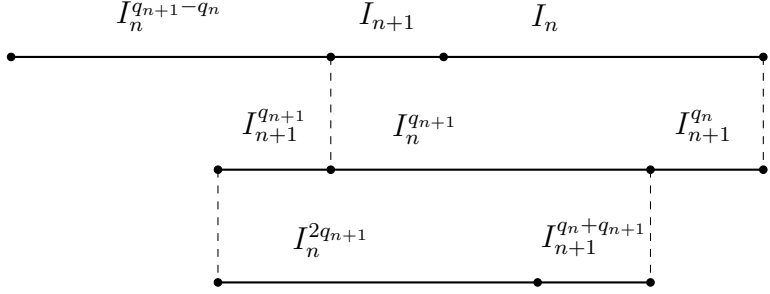
\begin{figure}[!ht]
    \centering
    \begin{tikzpicture}
    
    \begin{scope}[thick]
    
    \draw (-5, 2) -- (5, 2);
    \draw (-2.25, 0.5) -- (5, 0.5);
    \draw (-2.25, -1) -- (3.5, -1);
    
    \end{scope}
    
    \filldraw[color=black] (-5, 2) circle (0.05cm);
    \filldraw[color=black] (-0.75, 2) circle (0.05cm);
    \filldraw[color=black] (0.75, 2) circle (0.05cm);
    \filldraw[color=black] (5, 2) circle (0.05cm);
    \filldraw[color=black] (-2.25, 0.5) circle (0.05cm);
    \filldraw[color=black] (-0.75, 0.5) circle (0.05cm);
    \filldraw[color=black] (3.5, 0.5) circle (0.05cm);
    \filldraw[color=black] (5, 0.5) circle (0.05cm);
    \filldraw[color=black] (-2.25, -1) circle (0.05cm);
    \filldraw[color=black] (2, -1) circle (0.05cm);
    \filldraw[color=black] (3.5, -1) circle (0.05cm);
    
    \draw (-2.9, 2.2) node[anchor=south] {$I_n^{q_{n+1}-q_n}$};
    \draw (0, 2.2) node[anchor=south] {$I_{n+1}$};
    \draw (2.1, 2.2) node[anchor=south] {$I_n$};
    \draw (-1.5, 0.7) node[anchor=south] {$I_{n+1}^{q_{n+1}}$};
    \draw (0.5, 0.7) node[anchor=south] {$I_n^{q_{n+1}}$};
    \draw (4.2, 0.7) node[anchor=south] {$I_{n+1}^{q_n}$};
    \draw (-0.75, -0.8) node[anchor=south] {$I_n^{2 q_{n+1}}$};
    \draw (2.75, -0.8) node[anchor=south] {$I_{n+1}^{q_n+q_{n+1}}$};
    
    \begin{scope}[thin, dashed]
    
    \draw (-0.75, 2) -- (-0.75, 0.5);
    \draw (5, 2) -- (5, 0.5);
    \draw (-2.25, 0.5) -- (-2.25, -1);
    \draw (3.5, 0.5) -- (3.5, -1);
    
    \end{scope}
    
	\end{tikzpicture}
	\caption{Relative positions of the intervals $I_n$, $I_{n+1}$, $I_n^{q_{n+1}-q_n}$, $I_n^{q_{n+1}}$, $I_{n+1}^{q_n}$, $I_{n+1}^{q_{n+1}}$, $I_n^{2 q_{n+1}}$, $I_{n+1}^{q_n+q_{n+1}}$ when $a_{n+1} \geq 2$.}
	\label{fig: Last elements of Cn}
\end{figure}

Now, we use the Real Bounds to bound $\frac{\absol{I_n^{2 q_{n+1}}}}{\absol{I_{n+1}}}$. From \eqref{eq1 comp short long}, we have:
\begin{equation} \label{eq3 comp short long}
\begin{split}
\frac{\absol{I_n^{2 q_{n+1}}}}{\absol{I_{n+1}}} &= \frac{\absol{I_{n+1}^{q_{n+1}}}}{\absol{I_{n+1}}} + \frac{\absol{I_{n+1}}}{\absol{I_{n+1}}} + \frac{\absol{I_n}}{\absol{I_{n+1}}} - \frac{\absol{I_{n+1}^{q_n}}}{\absol{I_{n+1}}} - \frac{\absol{I_{n+1}^{q_n+q_{n+1}}}}{\absol{I_{n+1}}} \\
&< \frac{\absol{I_{n+1}^{q_{n+1}}}}{\absol{I_{n+1}}} + 1 + \frac{\absol{I_n}}{\absol{I_{n+1}}} \\
&\leq C + 1 + C = 1 + 2 C < 3 C
\end{split}
\end{equation}
since $I_{n+1}^{q_{n+1}}, I_{n+1}$ are adjacent atoms of $\mathcal{P}_{n+1}$, $I_n, I_{n+1}$ are adjacent atoms of $\mathcal{P}_n$ and $C > 1$.

Plugging \eqref{eq3 comp short long} into \eqref{eq2 comp short long}, we get
\begin{equation} \label{eq4 comp short long}
\rat(I_{n+1}) \leq (3 C)^s \rat(I_n^{2 q_{n+1}})
\end{equation}
which proves that $I_n^{2 q_{n+1}} \succeq I_{n+1}$.

\begin{figure}[!ht]
    \centering
    \begin{tikzpicture}
    
    \begin{scope}[thick]
    
    \draw (-5, 2) -- (5, 2);
    \draw (-2.25, 0.5) -- (5, 0.5);
    \draw (-2.25, -1) -- (3.5, -1);
    \draw (-2.25, -2.5) -- (1.5, -2.5);
    
    \end{scope}
    
     \begin{scope}[thin, dashed]
    
    \draw (-0.75, 2) -- (-0.75, -2.5);
    \draw (0.75, 2) -- (0.75, -2.5);
    \draw (5, 2) -- (5, 0.5);
    \draw (-2.25, 0.5) -- (-2.25, -1);
    \draw (3.5, 0.5) -- (3.5, -1);
    \draw (-2.25, -1) -- (-2.25, -2.5);
    \draw (0.25, -1) -- (0.25, -2.5);
    
    \end{scope}
    
    \filldraw[color=black] (-5, 2) circle (0.05cm);
    \filldraw[color=black] (-0.75, 2) circle (0.05cm);
    \filldraw[color=black] (0.75, 2) circle (0.05cm);
    \filldraw[color=black] (5, 2) circle (0.05cm);
    \filldraw[color=black] (-2.25, 0.5) circle (0.05cm);
    \filldraw[color=black] (-0.75, 0.5) circle (0.05cm);
    \filldraw[color=black] (3.5, 0.5) circle (0.05cm);
    \filldraw[color=black] (5, 0.5) circle (0.05cm);
    \filldraw[color=black] (-2.25, -1) circle (0.05cm);
    \filldraw[color=black] (0.25, -1) circle (0.05cm);
    \filldraw[color=black] (3.5, -1) circle (0.05cm);
    \filldraw[color=black] (-2.25, -2.5) circle (0.05cm);
    \filldraw[color=black] (-0.75, -2.5) circle (0.05cm);
    \filldraw[color=black] (0.25, -2.5) circle (0.05cm);
    \filldraw[color=black] (1.5, -2.5) circle (0.05cm);

    \draw (-2.9, 2.2) node[anchor=south] {$I_n^{q_{n+1}-q_n}$};
    \draw (0, 2.2) node[anchor=south] {$I_{n+1}$};
    \draw (2.1, 2.2) node[anchor=south] {$I_n$};
    \draw (-1.5, 0.7) node[anchor=south] {$I_{n+1}^{q_{n+1}}$};
    \draw (1.8, 0.7) node[anchor=south] {$I_n^{q_{n+1}}$};
    \draw (4.2, 0.7) node[anchor=south] {$I_{n+1}^{q_n}$};
    \draw (-1, -0.8) node[anchor=south] {$I_n^{2 q_{n+1}}$};
    \draw (2.75, -0.8) node[anchor=south] {$I_{n+1}^{q_n+q_{n+1}}$};
    \draw (-1.5, -2.3) node[anchor=south] {$I_{n+1}^{q_{n+1}}$};
    \draw (-0.2, -2.3) node[anchor=south] {$I_{n+2}^{q_{n+1}}$};
    \draw (1.2, -2.5) node[anchor=north] {$I_{n+2}^{q_{n+1} + q_{n+2}}$};
    
	\end{tikzpicture}
	\caption{Relative positions of the intervals $I_n$, $I_{n+1}$, $I_n^{q_{n+1}-q_n}$, $I_n^{q_{n+1}}$, $I_{n+1}^{q_n}$, $I_{n+1}^{q_{n+1}}$, $I_n^{2 q_{n+1}}$, $I_{n+1}^{q_n+q_{n+1}}$, $I_{n+2}$, $I_{n+2}^{q_{n+1}}$, $I_{n+2}^{q_{n+1} + q_{n+2}}$ when $a_{n+1} = a_{n+2} = 1$. By applying the real bounds and Lemma \ref{comparability of a bunch of images of atoms}, one can see all these intervals have comparable lengths.}
	\label{fig: Last elements of Cn 2}
\end{figure}
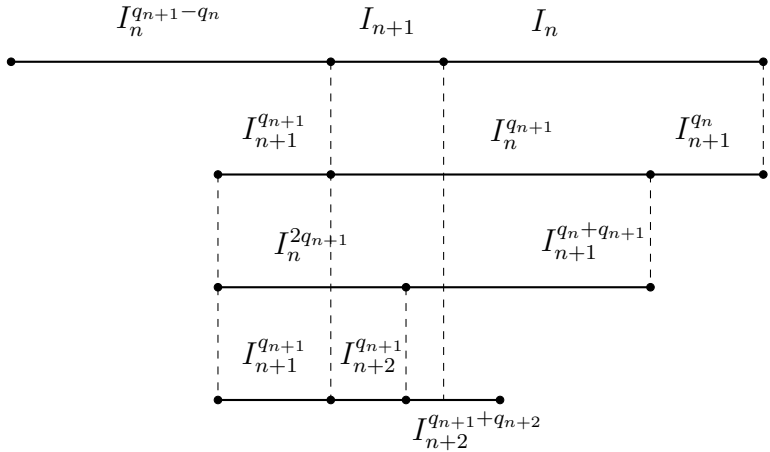

Finally, we address the case $a_{n+1} = a_{n+2} = 1$. Observe from figure \ref{fig: Last elements of Cn 2} that $I_n^{2 q_{n+1}} \supset I_{n+2}^{q_{n+1}}$ and $I_{n+1} \subset I_{n+2}^{q_{n+1}} \cup I_{n+2}^{q_{n+1} + q_{n+2}}$; from Lemma \ref{critical times dominate}, $I_{n+2}^{q_{n+1}} \succeq I_{n+2}^{q_{n+1} + q_{n+2}}$. By applying the Real Bounds and Lemma \ref{comparability of a bunch of images of atoms}, we obtain 
\[
	I_n^{2 q_{n+1}} \succeq I_{n+2}^{q_{n+1}} \succeq I_{n+2}^{q_{n+1}} \cup I_{n+2}^{q_{n+1} + q_{n+2}} \succeq I_{n+1}
\]
which finishes the proof.
\end{proof}

Though we expect Theorem \ref{general comparability of ratios} shall prove more useful in future works, for our purposes we will require the following stronger result, which clearly implies Theorem \ref{general comparability of ratios}.

\begin{theorem} \label{most general comparability of ratios} There exists a constant $B_2 = B_2(f, s) > 1$ with the following property. For any $n \geq n_0$ and $\Delta_1, \Delta_2 \in \mathcal{P}_n$, we have:
\begin{enumerate}
    \item[(a)] If $\Delta_1, \Delta_2$ are both long atoms or both short atoms of $\mathcal{P}_n$, then
    \begin{equation} \label{eq most comp ratios same type}
    B_2^{-1} \, \rat(\Delta_2) \leq \rat(\Delta_1) \leq B_2 \, \rat(\Delta_2) \, .
    \end{equation}
    \item[(b)] If $\Delta_1$ is a short atom and $\Delta_2$ is a long atom of $\mathcal{P}_n$, then
    \begin{equation} \label{eq most comp ratios diff type}
        \rat(\Delta_1) \leq B_2 \, \rat(\Delta_2) \, .
    \end{equation}
\end{enumerate}
Furthermore, if $a_{n+1} \geq 2$ or $a_{n+1} = a_{n+2} = 1$, then $\Delta_1, \Delta_2$ can be respectively replaced in the above inequalities by images $\Delta_1^{k_1}, \Delta_2^{k_2}$, $0 \leq k_1, k_2 \leq q_{n+1} + 1$.
\end{theorem}

\begin{proof} Observe that (b) is precisely the content of Lemma \ref{comparability of ratios short-long}. Moreover, since all short atoms of $\mathcal{P}_n$ become long atoms in $\mathcal{P}_{n+1}$, it suffices to prove (a) for long $\Delta_1, \Delta_2$.

Assume that $\Delta_1, \Delta_2$ are both long atoms of $\mathcal{P}_n$. We split the proof in two parts: (i) that $\Delta_1 \succeq \Delta_2$; and (ii) that $\Delta_1^{k_1} \succeq \Delta_2^{k_2}$ if $a_{n+1} \geq 2$ or $a_{n+1} = a_{n+2} = 1$. 

We first prove (i). It suffices to show that $\Delta_1$ $\rat$-dominates $\Delta_2$, since then the $\rat$-comparability of the two follows by simply interchanging $\Delta_1$ and $\Delta_2$. Once more, from Lemma \ref{critical times dominate}, $\Delta_1 \succeq I_n^{q_{n+1}}$ and $I_n \succeq \Delta_2$. Since $I_n \subset I_n^{q_{n+1}} \cup I_{n+1}^{q_n}$, $I_n^{q_{n+1}} \succeq I_{n+1}^{q_n}$ (by Lemma \ref{comparability of ratios short-long}) and these intervals have pairwise comparable lengths, we conclude that $I_n^{q_{n+1}} \succeq I_n$. This finishes the proof of (i).

The proof of (ii) is essentially the same: we have $\Delta_1^{k_1} \succeq I_n^{2 q_{n+1}}$ and $I_n \succeq \Delta_2^{k_2}$, so we need only show that $I_n^{2 q_{n+1}} \succeq I_n$. But $I_n \subset I_n^{2 q_{n+1}} \cup I_{n+1}^{q_n + q_{n+1}} \cup I_{n+1}^{q_n}$ and (from Lemma \ref{comparability of ratios short-long})
\begin{equation} \label{eq1 most comp}
I_n^{2 q_{n+1}} \succeq I_{n+1}^{q_n} \succeq I_{n+1}^{q_n + q_{n+1}}
\end{equation}
so, since these intervals have pairwise comparable lengths, we conclude that $I_n^{2 q_{n+1}} \succeq I_n$. This finishes the proof of (ii).
\end{proof}

\section{Ergodicity and uniqueness} \label{ergodic}

In this section, we prove that automorphic measures for multicritical circle maps with irrational rotation number are ergodic (Theorem \ref{ergodicity} below). As a consequence, we will obtain the uniqueness part of Theorem \ref{exist-unique} (we would like to remark that the non-flatness condition on each critical point of $f$ is crucial in order to have uniqueness, see \S \ref{ProofThmC} below). In particular, Lebesgue is the unique $f$-automorphic measure of exponent $1$ (Corollary \ref{lebesgue is unique}).

We further show that this uniqueness remains true (up to a scalar multiple) in the context of continuous linear functionals on $C^0(S^1)$ (Corollary \ref{signed uniqueness}). As we will see in \S \ref{distr}, Corollary \ref{signed uniqueness}, applied to Lebesgue measure ($s = 1$), is the core step towards proving Theorem \ref{no invar distr}.

\medskip

\subsection{The $\Gamma$ ratio} \label{gamma}

We first introduce a bit of notation. Fix some $s > 0$ and $\nu \in \mathcal{A}_s$. For an interval $I \subset S^1$ and a Borel $f$-invariant set $B \subset S^1$, we shall denote by $\Gamma(I; B)$ the ratio
\begin{equation} \label{eq def gamma}
\Gamma(I; B) \eqdef \frac{\nu(I \cap B)}{\nu(I)}.
\end{equation}

Observe that $\Gamma(I; B)$ can be expressed as the quotient of two $\rat$-ratios with respect to $I$ and different invariant sets: in the numerator, we take $B$ as the $f$-invariant set, and in the denominator, we take $S^1$ as the invariant set.

By direct analogy with Lemma \ref{comparability of ratios 1}, we thus obtain the following result:

\begin{lemma} \label{comparability of gammas 1} Let $n \geq 0$ and $\Delta \in \mathcal{P}_n$. Then, for any interval $I \subset \Delta^\ast$ and $0 \leq j < k \leq q_{n+1}$ such that the intervals $f^j(\chdois{\Delta}), \dots, f^{k-1}(\chdois{\Delta})$ do not contain any critical point of $f$, the following holds for all Borel $f$-invariant sets $B \subset S^1$:
\begin{equation} \label{eq comp gammas}
B_1^{-2s} \, \Gamma(I^j; B) \leq \Gamma(I^k; B) \leq B_1^{2s} \, \Gamma(I^j; B).
\end{equation}
\end{lemma}

\begin{definition} \label{critical time of type 2} Let $\Delta \in \mathcal{P}_n$, $0 \leq k < q_{n+1}$. We shall say that $k$ is a {\it critical time of type 2} for $\Delta$ if $f^k(\chdois{\Delta}) \cap \crit{f} \neq \varnothing$.
\end{definition}

Since $f$ has $N$ critical points $c_1, \dots, c_N$ and the collection $\listset{f^k(\chdois{\Delta})}_{k=0}^{q_{n+1}-1}$ has intersection multiplicity at most $8$ (Lemma \ref{bounding mult intersec}), it follows that, for any $n \geq 0$ and $\Delta \in \mathcal{P}_n$, there are at most $8N$ critical times of type 2 for $\Delta$.

\begin{remark} \label{one critical point from critical time 2} By a direct analogy with Remark \ref{one critical point from critical time}, there is some level $n_1 = n_1(f) \geq n_0(f) \in \N$, depending only on $f$, such that, for all $n \geq n_1$, $\Delta \in \mathcal{P}_n$ and $0 \leq k < q_{n+1}$ a critical time of type 2 for $\Delta$, we have that: (i) $f^k(\chdois{\Delta})$ contains a {\it single} critical point of $f$; and (ii) $f^k(\chdois{\Delta}) \subset U$, where $U$ is the interval about the critical point of $f$ in $f^k(\chdois{\Delta})$ from Proposition \ref{prop 5.4 dFG21}.
\end{remark}

With this terminology in mind, the following corollary is now an immediate consequence of Lemma \ref{comparability of gammas 1}:

\begin{coro} \label{comparability of gammas between critical times} Let $n \geq 0$, $\Delta \in \mathcal{P}_n$, and let $0 \leq k_1 < k_2 < \cdots < k_r < q_{n+1}$ be the critical times of type 2 for $\Delta$. Then, for any interval $I \subset \Delta$, the following holds for all Borel $f$-invariant sets $B \subset S^1$:
\begin{enumerate}
\item[(a)] For any $0 \leq j, \, \ell \leq k_1$, 
\begin{equation} \label{eq comp gammas between crits}
	B_1^{-2s} \, \Gamma(I^j; B) \leq \Gamma(I^\ell; B) \leq B_1^{2s} \, \Gamma(I^j; B) \, ;
\end{equation}
\item[(b)] \eqref{eq comp gammas between crits} also holds for any $k_i + 1 \leq j, \, \ell \leq k_{i+1}$, $1 \leq i < r$;
\item[(c)] \eqref{eq comp gammas between crits} also holds for $k_r + 1 \leq j, \, \ell \leq q_{n+1}$.
\end{enumerate}
\end{coro}

As we did in Lemma \ref{critical times dominate}, we now turn to the problem of understanding the behavior of $\Gamma$ close to a critical point of $f$. 

\begin{lemma} \label{bounding gamma of critical iterates} There exists a constant $B_3 = B_3(f, s) > 1$ with the following property. Let $n \geq n_1$ and $\Delta \in \mathcal{P}_n$. Assume that $n$ is such that either $a_{n+1} \geq 2$ or $a_{n+1} = a_{n+2} = 1$. Then the following holds for all Borel $f$-invariant sets $B \subset S^1$: if $0 \leq j < q_{n+1}$ is a critical time of type 2 for $\Delta$,
\begin{equation} \label{eq bounding critical iterate}
\Gamma(f^{j+1}(\Delta^\ast); B) \leq B_3 \; \Gamma(f^j(\Delta^\ast); B).
\end{equation}
\end{lemma}

\begin{remark} Note the hypotheses on the combinatorics of $f$ at level $n$: $a_{n+1} \geq 2$ or $a_{n+1} = a_{n+2} = 1$. These are necessary to allow for use of the sharpened version of Theorem \ref{most general comparability of ratios}.
\end{remark}

\begin{proof} We begin with the following observation (see Lemma \ref{comparability of a bunch of images of atoms}): for any one of the atoms $I$ of $\mathcal{P}_n$ that compose $\chdois{\Delta}$ and $0 \leq k \leq q_{n+1}$, we have $I^k \asymp f^k(\Delta^\ast)$. Further note that
\[
	((L)_3 \cup (L)_2 \cup \chum{L}) \cap (\chum{R} \cup (R)_2 \cup (R)_3) = \varnothing
\]
so (from Remark \ref{one critical point from critical time 2}) it cannot be that both $f^j((L)_3 \cup (L)_2 \cup \chum{L})$ and $f^j(\chum{R} \cup (R)_2 \cup (R)_3)$ contain a critical point of $f$. Since the other case is identical, we assume, without loss of generality, that
\[
	f^j((L)_3 \cup (L)_2 \cup \chum{L}) \cap \crit{f} = \varnothing \, .
\]

Now, since no two short atoms of $\mathcal{P}_n$ are adjacent, one of $L, (L)_2$ is a long atom; once more, since the other case is nearly identical, we assume, without loss of generality, that $(L)_2$ is a long atom of $\mathcal{P}_n$. Similarly, one of $\Delta, R$ is a long atom of $\mathcal{P}_n$, so we may assume, without loss of generality, that $R$ is also a long atom of $\mathcal{P}_n$.

It thus follows as a straightforward consequence of Theorem \ref{most general comparability of ratios} that $(L)_2^j \ratcomp f^j(\Delta^\ast)$ and $(L)_2^{j+1} \ratcomp f^{j+1}(\Delta^\ast)$. That is, there exists a constant $K_0 = K_0(f, s)$ such that
\begin{equation} \label{eq1 bounding critical iterate}
\begin{gathered}
K_0^{-1} \, \frac{\nu((L)_2^j \cap B)}{\absol{(L)_2^j}^s} \leq \frac{\nu(f^j(\Delta^\ast) \cap B)}{\absol{f^j(\Delta^\ast)}^s} \leq K_0 \, \frac{\nu((L)_2^j \cap B)}{\absol{(L)_2^j}^s}, \\
K_0^{-1} \, \frac{\nu((L)_2^j)}{\absol{(L)_2^j}^s} \leq \frac{\nu(f^j(\Delta^\ast))}{\absol{f^j(\Delta^\ast)}^s} \leq K_0 \, \frac{\nu((L)_2^j)}{\absol{(L)_2^j}^s}, \\
K_0^{-1} \, \frac{\nu((L)_2^{j+1} \cap B)}{\absol{(L)_2^{j+1}}^s} \leq \frac{\nu(f^{j+1}(\Delta^\ast) \cap B)}{\absol{f^{j+1}(\Delta^\ast)}^s} \leq K_0 \, \frac{\nu((L)_2^{j+1} \cap B)}{\absol{(L)_2^{j+1}}^s}, \\
K_0^{-1} \, \frac{\nu((L)_2^{j+1})}{\absol{(L)_2^{j+1}}^s} \leq \frac{\nu(f^{j+1}(\Delta^\ast))}{\absol{f^{j+1}(\Delta^\ast)}^s} \leq K_0 \, \frac{\nu((L)_2^{j+1})}{\absol{(L)_2^{j+1}}^s}.
\end{gathered}
\end{equation}
Therefore,
\begin{equation} \label{eq2 bounding critical iterate}
\begin{gathered}
K_0^{-2} \, \Gamma((L)_2^j; B) \leq \Gamma(f^j(\Delta^\ast); B) \leq K_0^2 \, \Gamma((L)_2^j; B), \\
K_0^{-2} \, \Gamma((L)_2^{j+1}; B) \leq \Gamma(f^{j+1}(\Delta^\ast); B) \leq K_0^2 \, \Gamma((L)_2^{j+1}; B).
\end{gathered}
\end{equation}

Now, since $f^j(\chum{(L)_2})$ contains no critical point of $f$, from Lemma \ref{Koebe for measure} and the mean value theorem,
\begin{equation} \label{eq3 bounding critical iterate}
B_0^{-1} \, \frac{\absol{(L)_2^{j+1}}}{\absol{(L)_2^j}} \leq D f(y) \leq B_0 \, \frac{\absol{(L)_2^{j+1}}}{\absol{(L)_2^j}} \quad\quad \ttt{for all } y \in (L)_2^j \, .
\end{equation}
Thus,
\begin{equation} \label{eq4 bounding critical iterate}
\Gamma((L)_2^{j+1}; B) = \frac{\int_{(L)_2^j \cap B} \, (D f)^s \, d\nu}{\int_{(L)_2^j} \, (D f)^s \, d\nu} \leq \frac{B_0^s \, \left( \frac{\absol{(L)_2^{j+1}}}{\absol{(L)_2^j}} \right)^s \, \int_{(L)_2^j \cap B} \, d\nu}{B_0^{-s} \, \left( \frac{\absol{(L)_2^{j+1}}}{\absol{(L)_2^j}} \right)^s \, \int_{(L)_2^j} \, d\nu} = B_0^{2s} \, \Gamma((L)_2^j; B).
\end{equation}
Combining equations \eqref{eq2 bounding critical iterate} and \eqref{eq4 bounding critical iterate}, we get:
\begin{equation} \label{eq5 bounding critical iterate}
\Gamma(f^{j+1}(\Delta^\ast); B) \leq K_0^2 \, \Gamma((L)_2^{j+1}; B) \leq K_0^2 \, B_0^{2s} \, \Gamma((L)_2^j; B) \leq B_0^{2s} \, K_0^4 \, \Gamma(f^j(\Delta^\ast); B)
\end{equation}
which is \eqref{eq bounding critical iterate}, with $B_3 \eqdef B_0^{2s} \, K_0^4$.
\end{proof}

Combining Corollary \ref{comparability of gammas between critical times} and Lemma \ref{bounding gamma of critical iterates}, we get the following.

\begin{coro} \label{gammas decrease}
There exists a constant $B_4 = B_4(f, s) > 1$ with the following property. Let $n \geq n_1$ and $\Delta \in \mathcal{P}_n$. Assume that $n$ is such that $a_{n+1} \geq 2$ or $a_{n+1} = a_{n+2} = 1$. Then, for any Borel $f$-invariant set $B \subset S^1$ and $0 \leq \ell_1 < \ell_2 \leq q_{n+1}$,
\begin{equation} \label{eq gammas decrease}
\Gamma(f^{\ell_2}(\Delta^\ast); B) \leq B_4 \; \Gamma(f^{\ell_1}(\Delta^\ast); B).
\end{equation}
\end{coro}

\smallskip

\subsection{Ergodicity} \label{proof ergodic}

We are now ready to prove the ergodicity of $\nu$ with respect to $f$.

\begin{theorem}[Ergodicity] \label{ergodicity} Let $s > 0$ and let $\nu \in \auto{s}$ be an automorphic measure of exponent $s$ for $f$. Then $\nu$ is ergodic with respect to $f$.
\end{theorem}

\begin{proof} Let $B \subset S^1$ be a Borel $f$-invariant set such that $\nu(B) < 1$. Our aim is to show that, in fact, $\nu(B) = 0$. For $x \in S^1$, consider the sequence
\[
	\mathcal{V}(x) = \setbuilder{\Delta^\ast}{n \geq 0, \Delta \in \mathcal{P}_n, x \in \Delta}
\]
of triples of adjacent atoms from the dynamical partitions $\mathcal{P}_n$ such that $x$ is contained in the central atom of the triple. As $n$ increases, the triples of level $n$ in this family shrink to $x$ while mantaining definite space on both sides (by the Real Bounds). If $x \in \Lambda$, then there is a unique atom $\Delta_n(x) \in \mathcal{P}_n$ that contains $x$, and $x$ is contained in its interior. Therefore, for $x \in \Lambda$, $\mathcal{V}(x)$ contains precisely one triple ($\Delta_n^\ast(x)$) of each level $n \geq 0$.

Since $\Lambda \setminus B$ has positive $\nu$-measure, we claim that for any given $\epsilon > 0$ there exist $x \in \Lambda \setminus B$ and $n_2 \geq n_1(f) \geq n_0(f)$ such that, for all $n \geq n_2$,
\begin{equation} \label{lebesgue density epsilon}
    \Gamma(\Delta_n^\ast(x); B) = \frac{\nu(\Delta_n^\ast(x) \cap B)}{\nu(\Delta_n^\ast(x))} < \epsilon\,.
\end{equation}
Indeed, note that since $\nu$ has no atoms and is supported on the whole circle (recall \S \ref{auto}), the map $h \colon S^1 \to S^1$ given by $h(x)=\int_{[0, x]} \, d\nu$ is a circle homeomorphism, which identifies the measure $\nu$ with Lebesgue measure on $S^1$. Thus, the existence of a point $x$ satisfying \eqref{lebesgue density epsilon} follows from the standard Lebesgue Density Theorem.

Now, some necessary distinctions depending on the combinatorics of $f$ must be made. If $\rho$ is the rotation number of $f$, $\rho = [a_0, a_1, \cdots]$, then either we have $a_n = 1$ for every sufficiently large $n$, or $a_n \geq 2$ occurs infinitely often. In any case, we can choose $n \geq n_2$ such that either $a_{n+1} = a_{n+2} = 1$ (in the first case) or $a_{n+1} \geq 2$ (in the latter case). We thus fix some level $n \geq n_2$ that satisfies one of these conditions, so that both \eqref{lebesgue density epsilon} and Corollary \ref{gammas decrease} hold simultaneously.

Observe that the collection $\listset{f^i(\Delta_n^\ast(x))}_{i=0}^{q_{n+1}}$ covers the circle and has intersection multiplicity at most $4$ (see Lemma \ref{bounding mult intersec}), so
\begin{equation} \label{bounding measure of B}
\begin{split}
    \nu(B) &= \nu \left( \bigcup_{i=0}^{q_{n+1}} \, f^i(\Delta_n^\ast(x)) \cap B \right) \leq \sum_{i=0}^{q_{n+1}} \, \nu(f^i(\Delta_n^\ast(x)) \cap B) \\
    &= \sum_{i=0}^{q_{n+1}} \, \nu(f^i(\Delta_n^\ast(x))) \, \Gamma(f^i(\Delta_n^\ast(x)); B)
\end{split}
\end{equation}
and furthermore,
\begin{equation} \label{mult intersect and measure}
\sum_{i=0}^{q_{n+1}} \, \nu(f^i(\Delta_n^\ast(x))) \leq 4
\end{equation}
(see Remark \ref{mult intersec and summability}).

From Corollary \ref{gammas decrease},
\begin{equation} \label{gamma bounds}
\Gamma(f^i(\Delta_n^\ast(x)); B) \leq B_4 \, \Gamma(\Delta_n^\ast(x)); B)
\end{equation}
so, plugging \eqref{gamma bounds}, \eqref{mult intersect and measure} and \eqref{lebesgue density epsilon} into \eqref{bounding measure of B}:
\begin{equation} \label{better bound on measure of B}
\begin{split}
\nu(B) &\leq \sum_{i=0}^{q_{n+1}} \, \nu(f^i(\Delta_n^\ast(x))) \, \Gamma(f^i(\Delta_n^\ast(x)); B) \leq B_4 \, \Gamma(\Delta_n^\ast(x); B) \, \sum_{i=0}^{q_{n+1}} \, \nu(f^i(\Delta_n^\ast(x))) \\
&\leq 4 \, B_4 \, \frac{\nu(\Delta_n^\ast(x) \cap B)}{\nu(\Delta_n^\ast(x))} < 4 \, B_4 \, \epsilon.
\end{split}
\end{equation}
By letting $\epsilon \to 0$, we get $\nu(B) = 0$.

Thus, there is no Borel $f$-invariant set $B \subset S^1$ such that $0 < \nu(B) < 1$, which proves that $\nu$ is ergodic.
\end{proof}

\subsection{Uniqueness} \label{proof unique}

We now show that ergodicity of $f$-automorphic measures of positive exponent implies that there is a unique $f$-automorphic measure of exponent $s$ for each $s > 0$. We further extend this uniqueness statement to finite signed measures (\emph{i.e.}, continuous linear functionals on $C^0(S^1)$).

\begin{proof}[Proof of Theorem \ref{exist-unique}, uniqueness part] Arguing by contradiction, suppose there is some $s > 0$ such that $\auto{s}$ contains two distinct measures, $\mu_s, \nu$. First, suppose $\nu \ll \mu_s$, and let $\psi \in L^1(\nu)$ be the Radon-Nikodym derivative
\[
	\psi = \frac{d\nu}{d\mu_s} \, .
\]
Then, as a simple calculation shows, $\psi \circ f = \psi$ $\mu_s$-almost everywhere, i.e. $\psi$ is $f$-invariant $\mu_s$-almost everywhere. But, from Theorem \ref{ergodicity}, $\mu_s$ is ergodic for $f$, so $\psi$ must be constant $\mu_s$-a.e. Since $\int_{S^1} \, \psi \, d\mu_s = 1$, we conclude that $\psi = 1$ $\mu_s$-almost everywhere, and thus, $\mu_s = \nu$, contradicting our assumption that $\mu_s, \nu$ are distinct.

Finally, if $\nu$ is not absolutely continuous with respect to $\mu_s$, let $\rho = \frac{1}{2}(\mu_s + \nu)$. Then $\rho \in \auto{s}$ (since $\auto{s}$ can be easily verified to be convex) and $\mu_s \ll \rho$, $\nu \ll \rho$. Thus, from the previous case, we must have $\mu_s = \rho = \nu$,
once more in contradiction to our assumption that $\mu_s, \nu$ are distinct.
\end{proof}

Thus, we have now given a complete proof of Theorem \ref{exist-unique}. Correspondingly, for $s > 0$, we shall denote the unique $f$-automorphic measure of exponent $s$ by $\mu_s$.

Since Lebesgue measure is always $f$-automorphic of exponent $1$, we have the following immediate consequence of Theorem \ref{exist-unique}:

\begin{coro} \label{lebesgue is unique} Lebesgue measure is the unique automorphic measure of exponent $1$ for~$f$.
\end{coro}

We now show that the uniqueness statement of Theorem \ref{exist-unique} (and in particular, Corollary \ref{lebesgue is unique}) remains true (up to a scalar multiple) in the context of continuous linear functionals on $C^0(S^1)$:

\begin{coro} \label{signed uniqueness} Let $s > 0$. If $L \in C^0(S^1)^\ast$ is such that 
\begin{equation} \label{eq signed}
	\innerprod{L}{\phi} = \innerprod{L}{(\phi \circ f)(D f)^s}
\end{equation}
for all $\phi \in C^0(S^1)$, then
\begin{equation} \label{eqeq signed}
L = \innerprod{L}{1} \, \mu_s \, .
\end{equation}
\end{coro}

\begin{proof} The proof is reproduced almost verbatim from ~\cite{DY}. From the Riesz Representation Theorem, there is a unique signed finite Radon measure $\nu$ on $S^1$ such that
\[
	\innerprod{L}{\phi} = \int_{S^1} \, \phi \, d\nu
\]
for all $\phi \in C^0(S^1)$. Therefore, it suffices to show that
\begin{equation} \label{eq1 signed}
	\nu = \nu(S^1) \, \mu_s \, .
\end{equation}

First, suppose $\nu$ is positive. Then $\tilde{\nu} \eqdef \frac{\nu}{\nu(S^1)} \in \auto{s}$, so
\[
	\nu = \nu(S^1) \, \tilde{\nu} = \nu(S^1) \, \mu_s \, .
\]

Now, for the general case, let $\nu = \nu_+ - \nu_-$ be the Jordan decomposition of $\nu$. We wish to show that
\begin{equation} \label{eq2 signed}
\begin{gathered}
\int_{S^1} \, \phi \, d\nu_+ = \int_{S^1} \, (\phi \circ f) (D f)^s \, d\nu_+ \, , \\
\int_{S^1} \, \phi \, d\nu_- = \int_{S^1} \, (\phi \circ f) (D f)^s \, d\nu_-
\end{gathered}
\end{equation}
for all $\phi \in C^0(S^1)$. Since the linear operator $U_s \colon C^0(S^1) \to C^0(S^1)$ given by
\[
	U_s(\psi) = (\psi \circ f) (D f)^s
\]
is positive, for all $\phi \in C^0(S^1)$ with $\phi \geq 0$, we have
\begin{equation} \label{eq3 signed}
\begin{split}
\int_{S^1} \, \phi \, d\nu_+ &= \sup_{0 \leq \psi \leq \phi} \, \int_{S^1} \, \psi \, d\nu = \sup_{0 \leq \psi \leq \phi} \,
\int_{S^1} \, (\psi \circ f) (D f)^s \, d\nu \\
&= \sup_{0 \leq \psi \leq U_s(\phi)} \, \int_{S^1} \, \psi \, d\nu = \int_{S^1} \, U_s(\phi) \, d\nu_+
\end{split}
\end{equation}
and similarly for $\nu_-$. It follows that \eqref{eq2 signed} holds for all continuous $\phi \geq 0$; by linearity, it must hold for all $\phi \in C^0(S^1)$. Thus, applying the first case to $\nu_+, \nu_-$, we have
\[
	\nu = \nu_+ - \nu_- = \nu_+(S^1) \, \mu_s - \nu_-(S^1) \, \mu_s = \nu(S^1) \, \mu_s
\]
which finishes the proof.
\end{proof}

\begin{remark} \label{remark Jordan decomp} It follows from \eqref{eqeq signed} that all linear functionals $L \in C^0(S^1)^\ast$ that satisfy \eqref{eq signed} have a {\it definite sign}: if $L = L_+ - L_-$ is its Jordan decomposition, then $L_+ = 0$ or $L_- = 0$, depending on whether $\innerprod{L}{1}$ is negative, positive or zero (and in this last case, $L = 0$).
\end{remark}

\section{Applications} \label{distr}

In this final section we prove Theorem \ref{no invar distr} (\S \ref{prelim inv distr}), Theorem \ref{TeoCexamples} (\S \ref{ProofThmC}) and Theorem \ref{denjoy-koksma} (\S \ref{ProofThmD}). With these purposes, we first review some basic results regarding invariant distributions for dynamical systems on compact manifolds.

\subsection{Cohomological equations and invariant distributions} \label{prelim inv distr}

Let $M$ be a compact smooth manifold. For integer $0 \leq r \leq \infty$, let $C^r(M)$ be the space of $C^r$ functions $u \colon M \to \R$, equipped with its $C^r$ topology. Recall that the $C^r$ topology turns $C^r(M)$ into a Banach space for finite $r$ and $C^\infty(M)$ into a Fr\'echet space, and that a distribution on $M$ is simply an element of the continuous dual space $C^\infty(M)^\ast$; we shall denote the space of distributions on $M$ by $\Distr{M}{}$, and the value of a distribution $T \in \Distr{M}{}$ acting on a function $u \in C^\infty(M)$ by $\innerprod{T}{u}$.

Suppose $T \in \Distr{M}{}$ and $0 \leq k < \infty$ are such that there exists $C > 0$ with
\[
	\absol{\innerprod{T}{u}}  \leq C \normof{u}_k \quad \forall \; u \in C^{\infty}(M).
\]
In this case, $T$ has a unique continuous extension $\tilde{T} \in C^k(M)^\ast$; we say that $T$ has {\it order at most $k$}. In fact, every $T \in C^k(M)^\ast$ is the unique continuous extension of a distribution on $M$. Denoting $C^k(M)^\ast$ by $\Distr{M}{k}$, we have the following chain of inclusions {\it modulo} unique extensions:
\[
	\Distr{M}{0} \hookrightarrow \Distr{M}{1} \hookrightarrow
	\cdots \hookrightarrow \Distr{M}{}.
\]
Observe that the Riesz Representation Theorem naturally identifies $\Distr{M}{0}$ with the space $\mathcal{M}(M)$ of signed finite Radon measures on $M$.

If $T \in \Distr{M}{}$ belongs to $\Distr{M}{k}$ for some finite $k$, we say that $T$ has finite order, and we define its order as the least such $k$. A noteworthy consequence of the compactness of $M$ is that all distributions on $M$ have finite order, i.e.
\[
	\Distr{M}{} = \bigcup_{k = 0}^\infty \, \Distr{M}{k} \, .
\]

Now, let $f \colon M \to M$ be a $C^r$ endomorphism of $M$, $0 \leq r \leq \infty$.

\begin{definition}[$C^\ell$-coboundary] \label{def coboundary} Let $0 \leq \ell \leq \infty$ and $\phi \in C^\ell(M)$. We say that $\phi$ is a {\it $C^\ell$-coboundary} for $f$ if the {\it cohomological equation}
\begin{equation} \label{CohomolEq}
u \circ f - u = \phi
\end{equation}
has a solution $u \in C^\ell(M)$.
\end{definition}

For integer $0 \leq \ell \leq \infty$, the set of $C^\ell$-coboundaries for $f$ forms a vector subspace of $C^\ell(M)$, which we shall denote by $\cobound\left(f, C^\ell(M)\right)$.

\begin{definition}[Invariant distribution] \label{invariant distr} We say that $T \in \Distr{M}{r}$ is {\it $f$-invariant} if 
\begin{equation} \label{eq invar distr}
	\innerprod{T}{u \circ f} = \innerprod{T}{u}
\end{equation}
for all $u \in C^\infty(M)$. 
\end{definition}

\begin{remark}\label{Sec8remauto} Let the manifold $M$ be the unit circle $S^1$. An $f$-automorphic measure $\nu$ of exponent $1$ naturally induces an $f$-invariant distribution $T\in\Distr{S^1}{1}$ by letting
\[
\innerprod{T}{u}=\int_{S^1}u'\,d\nu
\]
for all $u \in C^1(S^1)$. Indeed, note that
\[
\innerprod{T}{u \circ f}=\int_{S^1}(u \circ f)'\,d\nu=\int_{S^1}(u' \circ f)\,Df\,d\nu=\int_{S^1}u'\,d\nu=\innerprod{T}{u}.
\]
Of course, Lebesgue measure (which is automorphic of exponent $1$ for any $C^1$ circle homeo\-morphism) induces the null distribution $\innerprod{T}{u}=0$. However, automorphic measures of exponent $1$ different from Lebesgue provide non-trivial invariant distributions, see \S \ref{ProofThmC} below.
\end{remark}

For all integer $0 \leq k \leq r$, the set $\Distr{f}{k}$ of $f$-invariant distributions of order at most $k$ forms a vector subspace of $\Distr{M}{k}$. In fact, equations \eqref{CohomolEq} and \eqref{eq invar distr} (by unique extension to $C^k$ functions) identify $\Distr{f}{k}$ with the (continuous) annihilator of $\cobound\left(f, C^k(M)\right)$.

Thus, by the Hahn-Banach separation theorem,
\begin{equation} \label{KernelRep}
\closure{\cobound\left(f, C^k(M)\right)}{k} = \bigcap_{T \in \Distr{f}{k}} \ker T
\end{equation}
where $\closure{}{k}$ denotes closure in the $C^k$ topology.

Furthermore, we have the chain of inclusions
\[
	\Distr{f}{0} \hookrightarrow \Distr{f}{1} \hookrightarrow \cdots \hookrightarrow \Distr{f}{r}
\]
and also
\begin{equation} \label{invar distr have finite order}
\Distr{f}{r} = \bigcup_{k = 0}^r \, \Distr{f}{k} \, .
\end{equation}

The following proposition is a simple but crucial consequence of \eqref{KernelRep}.

\begin{prop} \label{DUE criterion} Let $f \colon M \to M$ be a $C^r$ endomorphism of a compact smooth manifold $M$, $0 \leq r \leq \infty$, and let $\mu$ be an $f$-invariant Radon probability measure on $M$. Let $0 \leq k \leq r$ be an integer. Then
\[
	\Distr{f}{k} = \R \mu
\]
if, and only if, the following holds. For any $\phi \in C^k(M)$ with $\int_{M} \, \phi \, d\mu = 0$, there is a sequence $\listset{\phi_n = u_n \circ f - u_n}_{n \geq 1} \subset \cobound\left(f, C^k(M)\right)$ of $C^k$-coboundaries for $f$ converging to $\phi$ in the $C^k$ topology.
\end{prop}

With the above criterion at hand, we are ready to prove that Theorem \ref{exist-unique} implies Theorem \ref{no invar distr}. The proof given below is taken almost verbatim from ~\cite{NT}. We reproduce it here for the sake of completeness as well as to indicate the points of the proof in which estimates depending on the bounded variation of $\log Df$ for $C^{1+\ttt{bv}}$-diffeomorphisms must be replaced by estimates suitable for multicritical circle maps and where results from ~\cite{DY} must be replaced by consequences of Theorem \ref{exist-unique}. From Proposition \ref{DUE criterion}, it suffices to show that Theorem \ref{exist-unique} implies the following lemma.

\begin{lemma} \label{no distr crucial lemma} Let $u \in C^1(S^1)$ have zero $\mu$-mean, that is, $\int_{S^1} \, u \, d\mu = 0$. Then there is a sequence $v_n$ of $C^1$ functions $S^1 \to \R$ such that
\begin{equation} \label{eq1 no distr crucial lemma}
v_n \circ f - v_n \longrightarrow u
\end{equation}
and
\begin{equation} \label{eq2 no distr crucial lemma}
(v_n' \circ f) D f - v_n' \longrightarrow u'
\end{equation}
uniformly.
\end{lemma}

Let $u \in C^1(S^1)$, $\int_{S^1} \, u \, d\mu = 0$, be fixed. The construction of the sequence $v_n$ from Lemma \ref{no distr crucial lemma} shall be derived as a consequence of the following fact.

\begin{prop} \label{no distr crucial prop} There exists a sequence $\listset{w_n}_{n \geq 1} \subset C^0(S^1)$ such that 
\begin{equation} \label{eq no distr crucial prop}
(w_n \circ f) D f - w_n \longrightarrow u'
\end{equation} 
uniformly and, for all $n \geq 1$, $\int_{S^1} \, w_n \, d\Leb = 0$.
\end{prop}

Indeed, assume that Proposition \ref{no distr crucial prop} is true, and for each $n \geq 1$, let $v_n \colon S^1 \to \R$ be defined by
\begin{equation*}
v_n(x) = \int_{[0, x]} \, w_n(y) \, dy\,,
\end{equation*}
where $[0, x]$ is the positively oriented closed circle interval with endpoints $0, x$.

Observe that, since $\int_{S^1} \, w_n \, d\Leb = 0$, $v_n$ is well-defined as a $\Z$-periodic function from $\R$ to $\R$ (\emph{i.e.}, $v_n \in C^0(S^1)$). Furthermore, $v_n \in C^1(S^1)$ and $v_n' = w_n$, so
\begin{equation*}
(v_n' \circ f) Df - v_n' = (w_n \circ f) Df - w_n \, .
\end{equation*}
Thus, Proposition \ref{no distr crucial prop} implies that the sequence $v_n$ we have just defined satisfies \eqref{eq2 no distr crucial lemma}. It remains to show that the $v_n \circ f - v_n$ also converge uniformly to $u$.

Well, for any $x \in S^1$, we have
\begin{equation} \label{eq1 prop implies lemma}
\begin{split}
v_n(f(x)) - v_n(x) - u(x) &= \int_{[0, f(x)]} \, w_n(y) \, dy - \int_{[0, x]} \, w_n(y) \, dy \\ 
&\quad - \left(u(0) + \int_{[0, x]} \, u'(y) \, dy \right) \\
&= \int_{[0, x]} \, \left[ \left( w_n(f(y)) D f(y) - w_n(y) \right) - u'(y) \right] \, dy - c_n
\end{split}
\end{equation}
where $c_n \eqdef u(0) - \int_{[0, f(0)]} \, w_n(y) \, dy$. Thus,
\begin{equation} \label{eq2 prop implies lemma}
\normof{(v_n \circ f - v_n + c_n) - u}_{C^0} \leq \normof{[(w_n \circ f) D f - w_n] - u'}_{C^0} \, .
\end{equation}

Now, from Proposition \ref{no distr crucial prop}, the right-hand side in \eqref{eq2 prop implies lemma} converges to $0$, so the sequence $\listset{v_n \circ f - v_n + c_n}_{n \geq 1}$ converges uniformly to $u$. Consequently, since $\mu$ is $f$-invariant,
\begin{equation}
c_n = \int_{S^1} \, (v_n \circ f - v_n + c_n) \, d\mu \longrightarrow \int_{S^1} \, u \, d\mu = 0
\end{equation}
so we conclude that \eqref{eq1 no distr crucial lemma} holds for the sequence $\listset{v_n}$ as well. This finishes the proof of Lemma \ref{no distr crucial lemma}, assuming Proposition \ref{no distr crucial prop}.

With the knowledge that Proposition \ref{no distr crucial prop} implies Lemma \ref{no distr crucial lemma} (which in turn implies Theorem \ref{no invar distr}), we now show that Theorem \ref{exist-unique} implies this proposition. First, a technical lemma, which consists of Proposition \ref{no distr crucial prop} in the special case $u=f-\Id$.

\begin{lemma} \label{no distr lemma prop} There exists a sequence $\listset{\hat{w}_k}_{k \geq 1} \subset C^0(S^1)$, with $(\hat{w}_k \circ f)Df-\hat{w}_k$ converging uniformly to $D f - 1$, such that, for all $k \geq 1$, $\int_{S^1} \, \hat{w}_k \, d\Leb = 0$.
\end{lemma}
\begin{proof} For $k \geq 1$, let $\hat{w}_k \in C^0(S^1)$ be defined by
\begin{equation} \label{eq1 no distr lemma prop}
\hat{w}_k = 1 - \frac{1}{q_k} \sum_{i = 0}^{q_k - 1} D f^i \, .
\end{equation}

Observe that
\[
\int_{S^1} \, D f^i \, d\Leb = 1\,,
\]
so
\begin{equation} \label{eq2 no distr lemma prop}
\int_{S^1} \, \hat{w}_k \, d\Leb = 1 - \frac{1}{q_k} \sum_{i = 0}^{q_k - 1} \, \int_{S^1} \, D f^i \, d\Leb = 1 - 1 = 0.
\end{equation}
Furthermore,
\begin{equation} \label{eq3 no distr lemma prop}
\begin{split}
&\quad \normof{\left[(\hat{w}_k \circ f) D f - \hat{w}_k\right] - (D f - 1)}_{C^0} \\
& = \normof{\left[\left( Df - \frac{1}{q_k} \sum_{i = 0}^{q_k - 1} (D f^i \circ f) D f \right) - 1 + \frac{1}{q_k} \sum_{i = 0}^{q_k - 1} D f^i \right] - D f + 1}_{C^0} \\
&= \frac{1}{q_k} \normof{1 - D f^{q_k}}_{C^0} \leq \frac{1 + \normof{D f^{q_k}}_{C^0}}{q_k} \, .
\end{split}
\end{equation}

Now, from Lemma \ref{lemma Dfqn},
\begin{equation} \label{eq4 no distr lemma prop}
\normof{\left[(\hat{w}_k \circ f) D f - \hat{w}_k\right] - (D f - 1)}_{C^0} \leq \frac{1 + \normof{D f^{q_k}}_{C^0}}{q_k} \leq 
\frac{1 + C_1}{q_k} \longrightarrow 0
\end{equation}
since $q_k \longrightarrow \infty$. Equations \eqref{eq4 no distr lemma prop} and \eqref{eq2 no distr lemma prop} prove the lemma.
\end{proof}

\begin{remark} \label{rem Dfqn distr} The main difference between the proofs of Lemma \ref{no distr lemma prop} above (in the critical case) and the corresponding lemma in ~\cite[p. 317]{NT} (in the diffeomorphism case) is the use of Lemma \ref{lemma Dfqn} to bound $\normof{D f^{q_n}}_{C^0}$, instead of the standard Denjoy inequality (see \cite[Section 3.2]{dFG:book}). Furthermore, in our case the $\hat{w}_k$ are not only continuous, but in fact $C^2$, since we require $f$ to be at least $C^3$.
\end{remark}

The proof of Proposition \ref{no distr crucial prop} given below depends essentially on the crucial fact that, if $L \in C^0(S^1)^\ast$ satisfies
\begin{equation} \label{no distr automorphic functional}
\innerprod{L}{(\phi \circ f) D f - \phi} = 0
\end{equation}
for all $\phi \in C^0(S^1)$, then $L$ is a scalar multiple of Lebesgue measure. In the diffeomorphism case, this fact is a consequence of ~\cite[Th\'eor\`eme 1]{DY}, while in the critical case, it follows from Theorem \ref{exist-unique} (recall Corollary \ref{signed uniqueness}).

\begin{proof}[Proof of Proposition \ref{no distr crucial prop}] The proof will result from two claims.

{\it Claim \#1:} There is a sequence $\listset{\bar{w}_n}_{n \geq 1} \subset C^0(S^1)$ such that
\begin{equation} \label{eq1 no distr prop}
(\bar{w}_n \circ f) D f - \bar{w}_n \longrightarrow u'
\end{equation}
uniformly. 

Indeed, consider the continuous linear operator $U_1 \colon C^0(S^1) \to C^0(S^1)$ given by
\begin{equation} \label{eq2 no distr prop}
U_1 \, w = (w \circ f) D f - w
\end{equation}
and let $M$ be the image of $U_1$. If no sequence $\bar{w}_n$ satisfying \eqref{eq1 no distr prop} exists, then $u' \notin \closure{M}{0}$, so the Hahn-Banach separation theorem implies the existence of a linear functional $L \in C^0(S^1)^\ast$ such that $L$ is identically null on $M$ and $\innerprod{L}{u'} = 1$. But the fact that $L$ is null on $M$ is easily seen to be equivalent to \eqref{no distr automorphic functional}, so $L$ must be a multiple of Lebesgue measure; this contradicts the fact that $\int_{S^1} \, u' \, d\Leb = 0$ (since $u$ is $\Z$-periodic). Thus, a sequence $\listset{\bar{w}_n}_{n \geq 1}$ satisfying \eqref{eq1 no distr prop} must exist. For each $n \geq 1$, let
\[
c_n \eqdef \int_{S^1} \, \bar{w}_n \, d\Leb\,, \quad\quad \tilde{w}_n \eqdef \bar{w}_n - c_n\,.
\]
and choose $k_n \in \N$ such that
\[
	\absol{c_n} \normof{\left[(\hat{w}_{k_n} \circ f) D f - \hat{w}_{k_n}\right] - (D f - 1)}_{C^0} \leq 2^{-n} 
\]
(Lemma \ref{no distr lemma prop} guarantees that this is possible). Finally, define
\begin{equation} \label{eq3 no distr prop}
w_n \eqdef \tilde{w}_n + c_n \hat{w}_{k_n} = \bar{w}_n + c_n (\hat{w}_{k_n} - 1) \, .
\end{equation}
Observe that $\int_{S^1} \, w_n \, d\Leb = 0$.

{\it Claim \# 2:}
\begin{equation} \label{eq4 no distr prop}
(w_n \circ f) D f - w_n \longrightarrow u'
\end{equation}
uniformly. Observe that proving claim \# 2 will finish the proof.

To prove the claim, let $\epsilon > 0$ be arbitrary, and choose $n_0 \in \N$ such that $2^{-n_0} < \frac{\epsilon}{2}$ and
\[
	\normof{(\bar{w}_n \circ f) D f - \bar{w}_n - u'}_{C^0} < \frac{\epsilon}{2}
\]
for all $n \geq n_0$.

Thus, for all $n \geq n_0$,
\begin{equation} \label{eq5 no distr prop}
\begin{split}
&\quad \normof{(w_n \circ f) D f - w_n - u'}_{C^0} \\
&\leq \normof{(\bar{w}_n \circ f) D f - \bar{w}_n - u'}_{C^0} + \absol{c_n} \normof{\left[(\hat{w}_{k_n} - 1) \circ f \right] D f - (\hat{w}_{k_n} - 1)}_{C^0} \\
&< \frac{\epsilon}{2} + \absol{c_n} \normof{\left[(\hat{w}_{k_n} \circ f) D f - \hat{w}_{k_n}\right] - (D f - 1)}_{C^0} \\
&\leq \frac{\epsilon}{2} + 2^{-n} < \frac{\epsilon}{2} + \frac{\epsilon}{2} = \epsilon
\end{split}
\end{equation}
which proves the claim.
\end{proof}

We have thus proved Theorem \ref{no invar distr}.

\subsection{Wandering intervals and invariant distributions}\label{ProofThmC} Before entering the proof of Theorem \ref{denjoy-koksma}, let us briefly explain why Theorem \ref{TeoCexamples} holds.

For any given irrational number $\rho\in(0,1)$, Hall was able to construct in \cite{Hall} a $C^{\infty}$ homeo\-morphism $f \colon S^1 \to S^1$, with rotation number $\rho(f)=\rho$, having a wandering interval~$I$ (\emph{i.e.}, $I$ is an open interval such that $f^n(I)$ is disjoint from $f^m(I)$ whenever $n \neq m$ in $\mathbb{Z}$). These examples, so-called \emph{Hall's examples}, present a single critical point $c$ which is \emph{flat}: the successive derivatives of $f$ (of all orders) vanish at $c$. Note that this critical point necessarily belongs to the invariant Cantor set of $f$ (otherwise, a $C^{\infty}$ perturbation supported on the wandering interval containing $c$ would produce a $C^{\infty}$ diffeomorphism with irrational rotation number and wandering intervals). In particular, $f^n \colon I \to f^n(I)$ is a diffeomorphism for all $n\in\mathbb{Z}$.

To these Hall's examples, we will apply the following general remark.

\begin{lemma}\label{leminhaconvmonot} Let $f \colon S^1 \to S^1$ be an orientation-preserving homeomorphism with a wandering interval $I \subset S^1$ such that $f^n \colon I \to f^n(I)$ is a $C^1$ diffeomorphism for all $n\in\mathbb{Z}$. Then, the series
\[
\sum_{n\in\mathbb{Z}}Df^n(x)
\]
is finite for Lebesgue almost every $x \in I$.
\end{lemma}

\begin{proof} Fix some $\ell\in\mathbb{N}$, and consider the set
\[
A_{\ell}=\big\{x \in I:\,\sum_{n\in\mathbb{Z}}Df^n(x)\geq\ell\big\}\,.
\]
We claim that $\Leb(A_{\ell}) \leq 1/\ell$, where $\Leb$ denotes the Lebesgue measure on $S^1$. Indeed, by the Monotone Convergence Theorem (recall here that $Df$ is non-negative on the whole circle),
\[
\int_{A_{\ell}}\,\sum_{n\in\mathbb{Z}}Df^n\,\,d\Leb\leq\int_{I}\,\sum_{n\in\mathbb{Z}}Df^n\,\,d\Leb=\sum_{n\in\mathbb{Z}}\,\int_{I}\,Df^n\,\,d\Leb\,.
\]
But
\[
\sum_{n\in\mathbb{Z}}\,\int_{I}\,Df^n\,\,d\Leb=\sum_{n\in\mathbb{Z}}\,\big|f^{n}(I)\big| \leq 1,
\]
since $I$ is a wandering interval and $f^n \colon I \to f^n(I)$ is a $C^1$ diffeomorphism for all $n\in\mathbb{Z}$. Thus, we have that
\[
\int_{A_{\ell}}\,\sum_{n\in\mathbb{Z}}Df^n\,\,d\Leb \leq 1\,,
\]
and then $\Leb(A_{\ell}) \leq 1/\ell$ for all $\ell \geq 1$, as claimed. The claim clearly implies the lemma.
\end{proof}

\medskip

\begin{proof}[Proof of Theorem \ref{TeoCexamples}] Let $f \colon S^1 \to S^1$ be a Hall's example as above, having a wandering interval $I \subset S^1$. Pick some $x \in I$ such that $S=\sum_{n\in\mathbb{Z}}Df^n(x)$ is finite (recall that, by Lemma \ref{leminhaconvmonot}, this series is finite for Lebesgue almost every $x$ outside the $f$-invariant Cantor set). Following \cite[Section 3.1]{DY}, we consider the probability measure
\[
\nu=\frac{1}{S}\,\sum_{n\in\mathbb{Z}}Df^n(x)\,\delta_{f^n(x)}\,.
\]
We see at once that $\nu$ is $f$-automorphic of exponent $1$ (note, in particular, that the \emph{uniqueness} part of Theorem \ref{exist-unique} also breaks down if we remove the non-flatness condition on the critical points of $f$). Now we consider $T\in\Distr{S^1}{1}$ given by
\[
\innerprod{T}{u}=\int_{S^1}u'\,d\nu.
\]
As explained in Remark \ref{Sec8remauto} above, the distribution $T$ is $f$-invariant. Finally, to prove that $T$ is \emph{not} a scalar multiple of the unique $f$-invariant probability measure $\mu$ is straightforward (compare \cite[Section 3]{NT}). Indeed, let $u \colon S^1 \to \mathbb{R}$ be of class $C^1$, supported on the wandering interval $I$, and such that $x$ is not a critical point of $u$. Then, on one hand, we have
\[
\innerprod{T}{u}=\int_{S^1}u'\,d\nu=\frac{1}{S}\,\,u'(x) \neq 0.
\]
On the other hand, since the support of $u$ is disjoint from the non-wandering set of $f$, we certainly have $\int_{S^1}u\,d\mu=0$. This finishes the proof of Theorem \ref{TeoCexamples}.
\end{proof}

\subsection{Denjoy-Koksma inequality improved}\label{ProofThmD} We finish this paper by proving that Theorem \ref{exist-unique} implies Theorem \ref{denjoy-koksma}. In fact, we will present two different proofs of Theorem \ref{denjoy-koksma}. The first proof works only when the observable $\phi$ is of class $C^1$, whereas the second works in the general case, \emph{i.e.}, when $\phi$ is absolutely continuous with respect to Lebesgue. The former follows \cite[pages 513--514]{AK}, and relies on the absence of invariant distributions of order $1$ (obtained in Theorem \ref{no invar distr}), while the latter follows \cite[pages 379--381]{NavasIMRN23}, and it only uses the ergodicity of the Lebesgue measure under a multicritical circle map (as established in Theorem \ref{ergodicity}). 

\begin{proof}[Proof of Theorem \ref{denjoy-koksma} for $C^1$ observables] For any given $\phi \in C^1(S^1)$ note that $\phi-\int_{S^1}\phi\,d\mu$ belongs to $\ker\mu$. Combining Theorem \ref{no invar distr} with Proposition \ref{DUE criterion} we have that, for any given $\varepsilon>0$, there exists $u \in C^1(S^1)$ such that
\[
\left\|\big(u \circ f-u\big)-\big(\phi-\int_{S^1}\phi\,d\mu\big)\right\|_{C^1}\leq\frac{\varepsilon}{2}\,.
\]
Let $\tilde{\phi} \in C^1(S^1)$ be given by $\tilde{\phi}=u \circ f-u+\int_{S^1}\phi\,d\mu$\,, so that $\big\|\tilde{\phi}-\phi\big\|_{C^1}\leq\varepsilon/2$. Since $\mu$ is $f$-invariant, we have $\int_{S^1}\tilde{\phi}\,d\mu=\int_{S^1}\phi\,d\mu$. Now for any given $x \in S^1$,
\begin{align*}
\left|\sum_{i=0}^{q_n-1}\phi\big(f^i(x)\big)-q_n\int_{S^1}\phi\,d\mu\right|&\leq\left|\sum_{i=0}^{q_n-1}\big(\phi-\tilde{\phi}\big)\big(f^i(x)\big)-q_n\int_{S^1}\big(\phi-\tilde{\phi}\big)\,d\mu\right|\\
&+\left|\sum_{i=0}^{q_n-1}\tilde{\phi}\big(f^i(x)\big)-q_n\int_{S^1}\tilde{\phi}\,d\mu\right|.
\end{align*}
Let us estimate both terms at the right side of this inequality. On one hand, by the standard Denjoy-Koksma inequality,
\[
\left|\sum_{i=0}^{q_n-1}\big(\phi-\tilde{\phi}\big)\big(f^i(x)\big)-q_n\int_{S^1}\big(\phi-\tilde{\phi}\big)\,d\mu\right|\leq\big\|\tilde{\phi}-\phi\big\|_{C^1}\leq\varepsilon/2\,.
\]
On the other hand,
\begin{align*}
\sum_{i=0}^{q_n-1}\tilde{\phi}\big(f^i(x)\big)-q_n\int_{S^1}\tilde{\phi}\,d\mu&=\sum_{i=0}^{q_n-1}\left[u\big(f^{i+1}(x)\big)-u\big(f^{i}(x)\big)+\int_{S^1}\phi\,d\mu\right]-q_n\int_{S^1}\tilde{\phi}\,d\mu\\
&=u\big(f^{q_n}(x)\big)-u(x)+q_n\left(\int_{S^1}\phi\,d\mu-\int_{S^1}\tilde{\phi}\,d\mu\right)\\
&=u\big(f^{q_n}(x)\big)-u(x)\,. 
\end{align*}
In particular,
\[
\left|\sum_{i=0}^{q_n-1}\tilde{\phi}\big(f^i(x)\big)-q_n\int_{S^1}\tilde{\phi}\,d\mu\right|\leq\|u\|_{C^1}\,\|f^{q_n}-\Id\|_{C^0}\,.
\]
By minimality of $f$, we can choose $n_0\in\mathbb{N}$ such that $\|u\|_{C^1}\,\|f^{q_n}-\Id\|_{C^0}<\varepsilon/2$ for all $n \geq n_0$. Therefore, $\left|\sum_{i=0}^{q_n-1}\phi\big(f^i(x)\big)-q_n\int_{S^1}\phi\,d\mu\right|<\varepsilon$ for all $x \in S^1$ and $n \geq n_0$. Since $\varepsilon$ is arbitrary, this finishes the proof.
\end{proof}

Let us now give a proof of Theorem \ref{denjoy-koksma} that works in general, following \cite[Section 2]{NavasIMRN23}. With this purpose, we will need the following lemma.

\begin{lemma} \label{DK key lemma} If $v \in L^1(\Leb)$ is a Lebesgue-integrable function on the circle such that $\int_{S^1} \, v \, d\Leb = 0$, then there exists a sequence $v_n$ of Lebesgue-integrable functions on the circle such that $\int_{S^1} \, v_n \, d\Leb = 0$ for all $n$ and
\[
	(v_n \circ f) D f - v_n \longrightarrow v
\]
in the $L^1$ sense.
\end{lemma}

\begin{proof} Consider the continuous linear operator $U \colon L^1(\Leb) \to L^1(\Leb)$ given by
$U w = (w \circ f) D f - w$, and let $M$ be the image of $U$. First, assume that $v \notin \closure{M}{}$; then, by the Hahn-Banach theorem, there exists $L \in L^1(\Leb)^\ast$ such that $L$ is identically null on $M$ and $\innerprod{L}{v} = 1$. By identification of $L^1(\Leb)^\ast$ with $L^\infty(\Leb)$, there exists an $L^\infty$ function $\phi$ such that
\[
	\innerprod{L}{w} = \int_{S^1} \, \phi \, w \, d\Leb
\]
for all $w \in L^1(\Leb)$. But then, for all $\psi = U w \in M$,
\begin{equation*}
\begin{split}
	0 &= \int_{S^1} \, \phi \, \left[ (w \circ f) D f - w \right] \, d\Leb
	= \int_{S^1} \, (\phi \circ f^{-1}) \, w \, d\Leb - \int_{S^1} \, \phi \, w \, d\Leb\,,
\end{split}
\end{equation*}
where we have used the fact that Lebesgue measure is $f$-automorphic of exponent $1$. 

Since the previous equality holds for all Lebesgue-integrable $w$, we conclude that $\phi$ is $f$-invariant $\Leb$-almost everywhere. As proved in Section \ref{proof ergodic}, any multicritical circle map with irrational rotation number is ergodic with respect to the Lebesgue measure. Therefore, $\phi$ must be almost everywhere constant, \emph{i.e.}, there exists some constant $\beta$ such that $\phi = \beta$ $\Leb$-almost everywhere. But then
\[
1 = \innerprod{L}{v} = \int_{S^1} \, v \, \phi \, d\Leb = \beta \, \int_{S^1} \, v \, d\Leb\,,
\]
contradicting the fact that $v \in \ker \Leb$. Thus, $v \in \closure{M}{}$, and there is some sequence $\bar w_n$ of $L^1$ functions for which 
\[
	(\bar w_n \circ f) D f - \bar w_n \longrightarrow v
\]
in the $L^1$ sense.

Now, recall the functions $\hat{w}_k$ from Lemma \ref{no distr lemma prop}: we have that
\[
	(\hat{w}_k \circ f) Df - \hat{w}_k \longrightarrow D f - 1
\]
uniformly. If we define $c_n \eqdef \int_{S^1} \, \bar w_n \, d\Leb$, then the desired sequence $w_n$ is given by
\[
	w_n \eqdef \bar w_n - c_n + c_n \, \hat{w}_{k_n}\,,
\]
where the $k_n$ are chosen as in the proof of Proposition \ref{no distr crucial prop}.
\end{proof}

\begin{proof}[Proof of Theorem \ref{denjoy-koksma}] For any given $\phi \in AC(S^1)$, we have that its derivative $v \eqdef D \phi$ exists $\Leb$-almost everywhere, and furthermore, $\int_{S^1} \, v \, d\Leb = 0$ (since $\phi$ is $\Z$-periodic). Let $\epsilon > 0$ be fixed. By Lemma \ref{DK key lemma}, there exists $w \in L^1(\Leb)$, $w \in \ker \Leb$, such that
\begin{equation} \label{eq1 DK}
	u \eqdef v - \left[ (w \circ f) D f - w \right]
\end{equation}
satisfies
\begin{equation} \label{eq2 DK}
	\int_{S^1} \, \absol{u} \, d\Leb \leq \frac{\epsilon}{2}.
\end{equation}

Now, let $\psi, \xi \in AC(S^1)$ be given by
\begin{equation*}
	\xi(x) = \int_{[0, x]} \, u \, d\Leb - \int_{S^1} \, \phi \, d\mu, \quad \psi(x) = \int_{[0, x]} \, w \, d\Leb\,,
\end{equation*}
so that $\int_{S^1} \, \xi \, d\mu = 0$ and
\begin{equation} \label{eq3 DK}
\xi = \phi - (\psi \circ f) + \psi - \int_{S^1} \, \phi \, d\mu\,,
\end{equation}
or equivalently,
\begin{equation} \label{eq4 DK}
\phi = \xi + (\psi \circ f) - \psi + \int_{S^1} \, \phi \, d\mu.
\end{equation}

By a telescoping sum, we have
\begin{equation} \label{eq5 DK}
\sum_{i=0}^{q_n-1} \, \phi \circ f^i - q_n \int_{S^1} \, \phi \, d\mu = \sum_{i=0}^{q_n-1} \, \xi \circ f^i + \psi \circ f^{q_n} - \psi.
\end{equation}
We now proceed to estimate both terms on the right hand side of the above equation. On the one hand, by the standard Denjoy-Koksma inequality,
\begin{equation} \label{eq6 DK}
\normof{\sum_{i=0}^{q_n-1} \, \xi \circ f^i}_{C^0} \leq \varia(\xi) \leq \normof{D \xi}_{L^1(\Leb)} = \normof{u}_{L^1(\Leb)} \leq \frac{\epsilon}{2}.
\end{equation}
On the other hand, by the minimality of $f$,
\begin{equation} \label{eq7 DK}
\normof{\psi \circ f^{q_n} - \psi}_{C^0} \leq \frac{\epsilon}{2}
\end{equation}
for sufficiently large $n$. By the triangle inequality, equations \eqref{eq5 DK}-\eqref{eq7 DK} together show that, for sufficiently large $n$,
\begin{equation} \label{eq8 DK}
\normof{\sum_{i=0}^{q_n-1} \, \phi \circ f^i - q_n \int_{S^1} \, \phi \, d\mu}_{C^0} \leq \epsilon.
\end{equation}
Since $\epsilon$ is arbitrary, this concludes the proof of Theorem \ref{denjoy-koksma}.
\end{proof}

\section*{Acknowledgements}
We would like to thank Andr\'es Navas and Michele Triestino for pointing to us the recent paper \cite{NavasIMRN23}. 


\bibliography{CircleDyn}

\providecommand{\bysame}{\leavevmode\hbox to3em{\hrulefill}\thinspace}
\providecommand{\MR}{\relax\ifhmode\unskip\space\fi MR }
\providecommand{\MRhref}[2]{%
  \href{http://www.ams.org/mathscinet-getitem?mr=#1}{#2}
}
\providecommand{\href}[2]{#2}
\begin{thebibliography}{10}

\bibitem{AK}
A.~Avila and A.~Kocsard, \emph{{Cohomological equations and invariant
  distributions for minimal circle diffeomorphisms}}, {Duke Math. J.}
  \textbf{158} (2011), no.~3, 501--536.

\bibitem{DY}
R.~Douady and J.-C. Yoccoz, \emph{{Nombre de rotation des diff\'eomorphismes du
  cercle et mesures automorphes (Rotation number of diffeomorphisms of the
  circle and automorphic measures)}}, Regul. Chaotic Dyn. \textbf{4} (1999),
  no.~4, 3--24.

\bibitem{EdF}
G.~Estevez and E.~de Faria, \emph{{Real bounds and quasisymmetric rigidity of
  multicritical circle maps}}, {Trans. Am. Math. Soc.} \textbf{370} (2018),
  no.~8, 5583--5616.

\bibitem{EG2023}
G.~Estevez and P.~Guarino, \emph{Renormalization of bicritical circle maps},
  Arnold Math. J. \textbf{9} (2023), no.~1, 69--104. \MR{4549077}

\bibitem{ESY2022}
G.~Estevez, D.~Smania, and M.~Yampolsky, \emph{Renormalization of analytic
  multicritical circle maps with bounded type rotation numbers}, Bull. Braz.
  Math. Soc. \textbf{53} (2022), no.~3, 1053--1071.

\bibitem{edsonETDS}
E.~de Faria, \emph{Asymptotic rigidity of scaling ratios for critical circle
  mappings}, Ergod. Th. \& Dynam. Sys. \textbf{19} (1999), 995--1035.
  \MR{1709428}

\bibitem{dFG2}
E.~de Faria and P.~Guarino, \emph{{Real bounds and Lyapunov exponents}},
  {Discrete Contin. Dyn. Syst.} \textbf{36} (2016), no.~4, 1957--1982.

\bibitem{dFG:book}
\bysame, \emph{{Dynamics of Circle Mappings}}, 33o Col\'oquio Brasileiro de
  Matem\'atica, IMPA Mathematical Publications, 2021.

\bibitem{dFG3}
\bysame, \emph{There are no {{\(\sigma\)}}-finite absolutely continuous
  invariant measures for multicritical circle maps}, Nonlinearity \textbf{34}
  (2021), no.~10, 6727--6749.

\bibitem{dFG1}
\bysame, \emph{Quasisymmetric orbit-flexibility of multicritical circle maps},
  Ergodic Theory Dynam. Systems \textbf{42} (2022), no.~11, 3271--3310.

\bibitem{dFdM1}
E.~de Faria and W.~de Melo, \emph{Rigidity of critical circle mappings. {I}},
  J. Eur. Math. Soc. (JEMS) \textbf{1} (1999), no.~4, 339--392.

\bibitem{dFdM2}
\bysame, \emph{Rigidity of critical circle mappings {II}}, J. Amer. Math. Soc.
  \textbf{13} (2000), 343--370. \MR{1711394}

\bibitem{GY}
I.~Gorbovickis and M.~Yampolsky, \emph{Rigidity, universality, and
  hyperbolicity of renormalization for critical circle maps with non-integer
  exponents}, Ergod. Th. \& Dynam. Sys. \textbf{40} (2020), 1282--1334.
  \MR{4082264}

\bibitem{GorYam2021}
I.~Gorbovickis and M.~Yampolsky, \emph{Rigidity of analytic and smooth bi-cubic
  multicritical circle maps with bounded type rotation numbers}, 2021.

\bibitem{GS93}
J.~Graczyk and G.~{\'S}wi{\c{a}}tek, \emph{Singular measures in circle
  dynamics}, Comm. Math. Phys. \textbf{157} (1993), 213--230. \MR{1244865}

\bibitem{GMdM}
P.~Guarino, M.~Martens, and W.~de Melo, \emph{Rigidity of critical circle
  maps}, Duke Math. J. \textbf{167} (2018), no.~11, 2125--2188.

\bibitem{GdM}
P.~Guarino and W.~de Melo, \emph{Rigidity of smooth critical circle maps}, J.
  Eur. Math. Soc. (JEMS) \textbf{19} (2017), no.~6, 1729--1783.

\bibitem{Hall}
G.~Hall, \emph{A {{\(C^\infty\)}} {Denjoy} counterexample}, Ergodic Theory Dyn.
  Syst. \textbf{1} (1981), 261--272.

\bibitem{He2}
M.~Herman, \emph{{Conjugaison quasi-symm\'etrique des hom\'eomorphismes du
  cercle \`a des rotations (Quasisymmetric conjugacy of analytic circle
  hoemomorphisms to rotations)}}, 1988.

\bibitem{Ka}
A.~Katok and E.~A.~Robinson Jr., \emph{Cocycles, cohomology and combinatorial
  constructions in ergodic theory}, pp.~107--173, Providence, RI: American
  Mathematical Society (AMS), 2001.

\bibitem{Kh1}
K.~Khanin, \emph{{Universal estimates for critical circle mappings}}, {Chaos}
  \textbf{1} (1991), no.~2, 181--186.

\bibitem{KT}
K.~Khanin and A.~Teplinsky, \emph{Robust rigidity for circle diffeomorphisms
  with singularities}, Invent. Math. \textbf{169} (2007), 193--218.
  \MR{2308853}

\bibitem{KY}
D.~Khmelev and M.~Yampolsky, \emph{The rigidity problem for analytic critical
  circle maps}, Mosc. Math. J. \textbf{6} (2006), 317--351. \MR{2270617}

\bibitem{dMP}
W.~de Melo and C.~Pugh, \emph{{The {{\(C^1\)}} Brunovsky Hypothesis}}, Journal
  of Differential Equations \textbf{113} (1994), no.~2, 300--337.

\bibitem{dMvS}
W.~de Melo and S.~van Strien, \emph{{One-dimensional dynamics}}, vol.~25,
  Berlin: Springer-Verlag, 1993.

\bibitem{NavasIMRN23}
A.~Navas, \emph{On conjugates and the asymptotic distortion of one-dimensional
  {$C^{1+bv}$} diffeomorphisms}, Int. Math. Res. Not. IMRN (2023), no.~1,
  372--405. \MR{4530112}

\bibitem{NT}
A.~Navas and M.~Triestino, \emph{On the invariant distributions of {{\(C^2\)}}
  circle diffeomorphisms of irrational rotation number}, Math. Z. \textbf{274}
  (2013), no.~1-2, 315--321.

\bibitem{Pat}
S.~J. Patterson, \emph{The limit set of a {F}uchsian group}, Acta Math.
  \textbf{136} (1976), no.~3-4, 241--273. \MR{450547}

\bibitem{Su2}
D.~Sullivan, \emph{The density at infinity of a discrete group of hyperbolic
  motions}, Inst. Hautes \'{E}tudes Sci. Publ. Math. (1979), no.~50, 171--202.
  \MR{556586}

\bibitem{Su1}
\bysame, \emph{Conformal dynamical systems}, Geometric dynamics ({R}io de
  {J}aneiro, 1981), Lecture Notes in Math., vol. 1007, Springer, Berlin, 1983,
  pp.~725--752. \MR{730296}

\bibitem{Sw1}
G.~{\'S}wi{\c{a}}tek, \emph{Rational rotation numbers for maps of the circle},
  Commun. Math. Phys. \textbf{119} (1988), no.~1, 109--128.

\bibitem{yampolsky1}
M.~Yampolsky, \emph{Complex bounds for renormalization of critical circle
  maps}, Ergod. Th. \& Dynam. Sys. \textbf{19} (1999), 227--257. \MR{1677153}

\bibitem{yampolsky2}
\bysame, \emph{The attractor of renormalization and rigidity of towers of
  critical circle maps}, Comm. Math. Phys. \textbf{218} (2001), 537--568.
  \MR{1828852}

\bibitem{yampolsky3}
\bysame, \emph{Hyperbolicity of renormalization of critical circle maps}, Publ.
  Math. IHES \textbf{96} (2002), 1--41. \MR{1985030}

\bibitem{yampolsky4}
\bysame, \emph{Renormalization horseshoe for critical circle maps}, Comm. Math.
  Phys. \textbf{240} (2003), 75--96. \MR{2004980}

\bibitem{Yam2019}
\bysame, \emph{Renormalization of bi-cubic circle maps}, C. R. Math. Acad. Sci.
  Soc. R. Can. \textbf{41} (2019), no.~4, 57--83.

\bibitem{Yo3}
J.-C. Yoccoz, \emph{{Il n'y a pas de contre-exemple de Denjoy analytique.
  (There are no analytic Denjoy counterexamples)}}, {C. R. Acad. Sci., Paris,
  S\'er. I} \textbf{298} (1984), 141--144.

\end{thebibliography}

\bibliographystyle{amsplain}


\end{document}